\documentclass[11pt]{article}
\usepackage{latexsym,amsmath, mathrsfs,amssymb,tabu,
	CJKutf8,graphicx,bm,blkarray,delarray,float,color,soul,extarrows, ulem, framed}
	\usepackage[thicklines]{cancel}
\usepackage{hyperref}
\hypersetup{hypertex=true, 
				colorlinks=true, 
					linkcolor=blue,
					anchorcolor=blue,
					citecolor=blue}
\newtheorem{lemma}{\it Lemma}
\newtheorem{theorem}{\it Theorem}
\newtheorem{proof}{\it Proof}

\begin{document}
\setulcolor{red} 
\setstcolor{red} 
\sethlcolor{yellow} 

\title{\bf Error estimate for the first order 
energy stable \\ scheme of ${\bf Q}$-tensor nematic model}
\author{Jin Huang, Xiao Li and Guanghua Ji\footnote{Corresponding author. \\Emails: 202331130037@mail.bnu.edu.cn(J. Huang), lixiao@bnu.edu.cn (X. Li), ghji@bnu.edu.cn (G. Ji).} \\Laboratory of Mathematics and Complex Systems (Ministry of Education), \\School of Mathematical Sciences, Beijing Normal University, Beijing 100875, P.R. China.} 
     \date{}
 \maketitle

\begin{abstract}
We present rigorous error estimates towards a first-order unconditionally energy stable scheme 
	designed for 3D hydrodynamic ${\bf Q}$-tensor model of nematic liquid crystals.
This scheme combines the scalar auxiliary variable (SAV), stabilization and projection method together.
The unique solvability and energy dissipation of the scheme are proved.
We further derive the boundness of numerical solution in $L^{\infty}$ norm 
with mathematical deduction. 
Then, 
we can give the rigorous error estimate of order $O(\delta t)$ in the sense of $L^2$ norm, 
where $\delta t$ is the time step. 
Finally, we give some numerical simulations to demonstrate the theoretical analysis.
\end{abstract}

\section{Introduction}
Liquid crystals present a state of matter with properties intermediate between those of liquids and crystalline solids. 
Nematic liquid crystals typically exhibit molecular orientational order but lack positional order. 
Several models of nematic liquid crystals are described using velocity and pressure $({\bf u},p)$ and a director vector ${\bf d}$ \cite{de_Gennes_1995_The, Leslie_FM_1968_Some}. 
A key limitation of the unit vector model is
	its singularities at point defects, 
	as well as its failure to capture line defects 
	and biaxiality in the original Ericksen-Leslie theory.
An alternative approach is use
 	a symmetric and traceless matrix ${\bf Q}\in R^{3\times 3}$ 
 	to describe the orientation of liquid crystals.
In fact,
 	${\bf Q}$ measures the deviation of the second order moment tensor from its isotropic value.
We consider the general Landau-De Gennes free energy functional
\begin{equation}\label{eq:model_free_energy}
E({\bf Q}) = \int_{\Omega} \left( 
		\frac{K}{2} |\nabla {\bf Q}|^2 + F_B({\bf Q}) 
		\right) d{\bf x},
\end{equation}
where the first two terms in the integral represent the kinetic energy and elastic energy, respectively. 
	$K$ is a material-dependent elastic constant, 
	$F_B$ is the bulk free energy density,
\begin{align*}
F_B({\bf Q}) = \frac{\alpha}{2} tr({\bf Q}^2) + \frac{\beta}{3}tr({\bf Q}^3) + \frac{\gamma}{4} tr^2({\bf Q}^2).
\end{align*}
and $\alpha$, $\beta$ and $\gamma > 0$ are material-dependent and temperature-dependent constants.

According to \cite{Beris_AN_1994_Thermodynamics, Wang_Q_2002_A, Zhao_J_2016_Semi}, 
	the non-dimensional governing equations of nematic liquid
	crystal flows with hydrodynamics have the following form
\begin{equation}\label{eq:model_equations}
\begin{aligned}
&{\bf Q}_t + {\bf u}\cdot \nabla {\bf Q} - S(\nabla {\bf u}, {\bf Q}) = M {\bf G}, 
\\
&
{\bf G} = -\frac{\delta E({\bf Q})}{\delta {\bf Q}} = 
	K \Delta {\bf Q} - f_B({\bf Q}), 
\\
&{\bf u}_t + {\bf u}\cdot \nabla {\bf u} =  \eta \Delta {\bf u} - \nabla p 
	+ \nabla \cdot \sigma({\bf Q}, {\bf G}) - \nabla {\bf Q}:{\bf G},
\\
& \nabla \cdot {\bf u} = 0.
\end{aligned}
\end{equation}
where
\begin{equation}
\begin{aligned}
& 
S(\nabla {\bf u}, {\bf Q}) = {\bf W}\cdot {\bf Q} - {\bf Q} \cdot {\bf W}
		+ a({\bf Q}\cdot {\bf D} + {\bf D} \cdot {\bf Q}) + \frac{2a}{3}({\bf D} - \frac{\nabla \cdot {\bf u}}{3}{\bf I})
		- 2a({\bf D} : {\bf Q})\left( {\bf Q} + \frac{\bf I}{3}\right),
\\ &
\sigma({\bf Q}, {\bf G}) = {\bf Q}\cdot {\bf G} - {\bf G} \cdot {\bf Q}
		- a({\bf G}\cdot {\bf Q} + {\bf Q} \cdot {\bf G}) - \frac{2a}{3}{\bf G} 
		+ 2a({\bf Q} : {\bf G})\left( {\bf Q} + \frac{\bf I}{3}\right),
\\ &
f_B({\bf Q}) = \alpha {\bf Q} + \beta \left( {\bf Q}^2 - \frac{tr({\bf Q}^2)}{3} {\bf I}\right)
	+ \gamma tr({\bf Q}^2){\bf Q},
\end{aligned}
\end{equation} 
${\bf D} = \frac{\nabla {\bf u} + \nabla {\bf u}^T}{2}$
	and 
	${\bf W} = \frac{\nabla {\bf u} - \nabla {\bf u}^T}{2}$
	are respectively, the rate of strain and vorticity tensors.
Moreover, 
	the first two terms in $S(\nabla {\bf u}, {\bf Q})$ together with the material derivative of
${\bf Q}$ define the Gordon-Schowalter derivative, 
${\bf G}$ is the molecular field, 
${\sigma}({\bf Q}, {\bf G})$ the elastic stress tensor, 
and $a \in [-1,1]$ a geometric parameter of the nematic liquid crystal molecule
\cite{Wang_Q_2002_A}.
The initial conditions of the system are
\begin{align*}
{\bf u}({\bf x},0) = {\bf u}_0({\bf x}), \qquad 
	 {\bf Q}({\bf x},0) = {\bf Q}_0({\bf x}).
\end{align*}
The following boundary conditions can be used:
\begin{itemize}
\item[1.] ${\bf u}$ and ${\bf Q}$ are periodic on $\partial \Omega$.

\item[2.] ${\bf u}|_{\partial \Omega} = 0$, 
			${\bf Q}|_{\partial \Omega} = {\bf Q}^0$ 
			or 
			$\partial_n {\bf Q}|_{\partial \Omega} = 0$. 
\end{itemize}
where ${\bf Q}_0$ and ${\bf Q}^0$ present the initial and boundary condition of the problem,
respectively.

There are some mathematical studies
\cite{Cavaterra_C_2016_Global, Paicu_M_2012_Energy, Fan_J_2013_Regularity, 
		Guillén-González_F_2014_Weak, Guillén-González_F_2015_Weak} 
		and numerical schemes \cite{MacDonald_CS_2015_Efficient, Ramage_A_2016_Computational} about the model.
Despite the aforementioned numerical works are developed to simulate flow behavior of nematic liquid crystals with the ${\bf Q}$-tensor model, 
there has been less systematic effort to analyze the numerical scheme 
	employed to solve the model equations, 
	especially in exploring the variational and dissipative property of the model.
Cai, Shen and Xu \cite{Cai_Y_2017_Stable} proposed a first-order nonlinear scheme 
    for a 2D dynamic ${\bf Q}$-tensor model of nematic liquid crystals 
	using a stabilizing technique\cite{J_Shen_2010_discrete, J_Shen_2015_decoupled},
	and established the maximum principle and convergence analysis. 
However, the scheme they constructed cannot guarantee unconditional energy stability. 
For the hydrodynamic {\bf Q}-tensor model,
Zhao et al.\cite{Zhao_J_2016_Semi} developed a second-order, coupled, nonlinear scheme 
	using the convex splitting strategy\cite{DJ_Eyre_1998_unconditionally}.
Later, 
they\cite{Zhao_J_2017_novel} developed a linear, second-order, 
	and unconditionally energy stable numerical scheme
	using an "IEQ" strategy\cite{X_Yang_2016_linear, X_Yang_2017_numerical}.
Neglecting the influence of flow fields,
Yue et al.\cite{Gudibanda_VM_2022_Convergence, Yue_Y_2022_On} 
	presented a fully discrete convergent finite difference scheme 
	for the 3D $Q$-tensor model of liquid crystals based on the IEQ method. 
They proved the stability of the scheme 
    and further obtained a uniform $H^2$ estimate for the numerical solutions.
Ji \cite{Ji_G_2020_BDF2} constructed a BDF2 linear energy-stable numerical scheme
	for the hydrodynamic $Q$-tensor model
  	by following the "SAV" startegy\cite{J_Shen_2018_scalar, J_Shen_2019_new, C_Chen_2019_fast}.


To the best of our knowledge, 
error estimates for the hydrodynamic $Q$-tensor model are still unavailable.
Therefore, we aim to give detailed mathematical analysis 
	for a first order unconditionally energy stable scheme. 
The rest of this paper is organised as follows. 
In Section \ref{sec:Numerical_schemes}, 
we use the SAV method combind with stabilization terms 
	to construct a semi-discrete unconditionally energy-stable scheme 
	\eqref{eq:first_order_scheme_Q}-\eqref{eq:first_order_scheme_nabla_u}
	for the coupled CH-NS system.
We also prove that the solution of the corresponding system is unique. 
In Section \ref{sec:Error_estimate}, 
	we derive the error estimates for phase function, velocity field.
Finally, some numerical examples are presented to verify analytical results 
	in Section \ref{Sec:Numerical_examples}.


\section{Numerical schemes}
\label{sec:Numerical_schemes}
In this section, 
	we construct a linear and unconditionally energy stable scheme
to solve the system \ref{eq:model_equations}. 
We use the SAV approach combined with stabilization terms 
based on the backward differentiation formula for temporal discretization. 
Before starting the following analysis, we will introduce some important notations and definitions that will be used throughout the paper.

For any tensor ${\bf M}_1, {\bf M}_2 \in \mathcal{R}^{3 \times 3}$, 
	vector functions ${\bf u}, {\bf v} \in \mathcal{R}^{3}$ 
	and scalar functions $f, g \in \mathcal{R}$ in $L^2 (\Omega)$, 
we define ${\bf M}_1 : {\bf M}_2 = tr({\bf M}_1 {\bf M}_2)$
and the following $L^2$ inner products:
\begin{align*}
\left( {\bf M}_1, {\bf M}_2 \right) 
	= \int_{\Omega} \sum_{i,j=1}^3 
		\left( {\bf M}_1 \right)_{ij} \left( {\bf M}_2 \right)_{ij}
	  d{\bf x}, 
\quad
\left( {\bf u}, {\bf v} \right) 
	= \int_{\Omega} \sum_{i=1}^3 
			{\bf u}_i {\bf v}_i 
	  d{\bf x}, 
\quad
\left( f, g \right) 
	= \int_{\Omega} fg d{\bf x}. 
\end{align*}
The $L^2$ norms are given by
\begin{align*}
\left\| {\bf M} \right\|^2 = \left( {\bf M}, {\bf M} \right), 
\quad
\left\| {\bf u} \right\|^2 = \left( {\bf u}, {\bf u} \right),
\quad
\left\| f \right\|^2 = \left( f, f \right).
\end{align*}
Moreover, we define the space $\mathcal{H}^s = \{ {\bf u} \in H^s, \nabla \cdot {\bf u} = 0\}$.

\subsection{Equivalent System}
We first reformulate model \eqref{eq:model_equations} into an equivalent form 
	using the SAV method.	
	
Let $\mathcal{E}_1 = \int_{\Omega} (F_B -\frac{S_Q}{2}tr({\bf Q}^2)) dx + C_0$, 
we introduce the scalar auxiliary variable $r = \sqrt{\mathcal{E}_1}$, 
	where $C_0$ is a constant which ensures $\mathcal{E}_1 \geq 0$
		and
		$\frac{S_Q}{2}tr({\bf Q}^2)$ is a stabilization term.   
Then, the free energy \eqref{eq:model_free_energy} can be rewritten as
\begin{equation}\label{eq:modified_energy}
E = \int_{\Omega} \left( \frac{S_Q}{2}tr({\bf Q}^2) + \frac{K}{2}|\nabla {\bf Q}|^2 \right) dx + r^2 - C_0.
\end{equation}
Moreover, we can write the Eq.\eqref{eq:model_equations} as
\begin{equation}\label{eq:equal_model}
\begin{aligned}
&{\bf Q}_t + {\bf u}\cdot \nabla {\bf Q} - S(\nabla {\bf u}, {\bf Q}) = M {\bf G}, 
\\
&
{\bf G} = K\Delta {\bf Q} -S_Q {\bf Q} - r{\bf V},
\\ 
&
r_t = \frac{1}{2}({\bf V}, {\bf Q}_t),
\\
&{\bf u}_t + {\bf u}\cdot \nabla {\bf u} =  \eta \Delta {\bf u} - \nabla p 
	+ \nabla \cdot \sigma({\bf Q}, {\bf G}) - \nabla {\bf Q}:{\bf G},
\\
& \nabla \cdot {\bf u} = 0
\end{aligned}
\end{equation}
where ${\bf V} = \frac{f_B({\bf Q}) - S_Q {\bf Q}}{\sqrt{\mathcal{E}_1({\bf Q})}}
:= \frac{g({\bf Q})}{\sqrt{\mathcal{E}_1({\bf Q})}}$.

\subsection{Time discretization, first-order SAV scheme}
In this subsection, 
we now develop a first order semi-discrete scheme
	using the SAV approach for the newly transformed system\eqref{eq:equal_model}. 
Let $\delta t > 0$ be the time step size 
	and $t^n = n\delta t$ for $n \geq 0$.
We denote the approximation of the variable $f({\bf x}, t)$
	at $t^n$ as $f^n({\bf x})$.

The first-order SAV scheme 
	for the hydrodynamic coupled model \eqref{eq:equal_model} 
	is given as follows.
	
{\bf Step 1.} Determine $({\bf Q}^{n+1}, r^{n+1}, \tilde{\bf u}^{n+1})$ 
from the equations
\begin{align}
&\frac{{\bf Q}^{n+1} - {\bf Q}^{n}}{\delta t} + \tilde{{\bf u}}^{n+1} \cdot \nabla {\bf Q}^{n}
	- S(\nabla \tilde{{\bf u}}^{n+1}, {\bf Q}^{n}) = M {\bf G}^{n+1},
\label{eq:first_order_scheme_Q}
\\ &
{\bf G}^{n+1} = K \Delta {\bf Q}^{n+1} - S_Q {\bf Q}^{n+1} -  r^{n+1}{\bf V}^n,
\label{eq:first_order_scheme_G}
\\ &
\frac{r^{n+1} - r^n}{\delta t} = \frac{1}{2}\int_{\Omega} {\bf V}^n :\frac{{\bf Q}^{n+1} - {\bf Q}^{n}}{\delta t} dx,
\label{eq:first_order_scheme_r}
\\ &
\frac{\tilde{\bf u}^{n+1} - {\bf u}^{n}}{\delta t} + ({\bf u}^n \cdot \nabla )\tilde{\bf u}^{n+1}
	= \eta \Delta \tilde{\bf u}^{n+1} - \nabla p^n + \nabla \cdot \sigma({\bf Q}^n, {\bf G}^{n+1})
	- \nabla {\bf Q}^n : {\bf G}^{n+1},
\label{eq:first_order_scheme_uhat}
\end{align}
where $S^n = S(\nabla {\bf u}^n, {\bf Q}^n)$, 
		$\sigma^n = \sigma({\bf Q}^n, {\bf H}^n)$, 
		${\bf V}^n = {\bf V}({\bf Q}^n)$.
The boundary conditions are either
\begin{align*}
(i) \ \text{periodic;}
\
(ii) \ \tilde{\bf u}^{n+1} |_{\partial \Omega} = 0, \
		{\bf Q}^{n+1} |_{\partial \Omega} = {\bf Q}^0 \
		or \
		\partial_{\bf n}{\bf Q}^{n+1} |_{\partial \Omega} = 0.
\end{align*}
		
{\bf Step 2.} Determine $({\bf u}^{n+1}, p^{n+1})$ 
from the equations
\begin{align}
&
\frac{{\bf u}^{n+1} - \tilde{\bf u}^{n+1} }{\delta t} + \nabla (p^{n+1} - p^n) = 0, 
\label{eq:first_order_scheme_u_p}
\\ &
\nabla \cdot {\bf u}^{n+1} = 0.
\label{eq:first_order_scheme_nabla_u}
\end{align}
where the boundary conditions are either $(i)$ periodic or $(ii)$ ${\bf u}^{n+1} \cdot {\bf n}|_{\partial \Omega} = 0$.

\subsection{Energy stability}
We next analyze the energy stability of the scheme defined in \eqref{eq:first_order_scheme_Q}-\eqref{eq:first_order_scheme_r}.

\begin{theorem}\label{theorem:discrete_energy_stable}
The scheme \eqref{eq:first_order_scheme_Q}-\eqref{eq:first_order_scheme_r} is unconditionally energy stable in the sense that
\begin{align}
\mathcal{E}^{n+1} - \mathcal{E}^n \leq -  \eta \delta t \left\| \nabla \tilde{\bf u}^{n+1}\right\|^2 
- M \delta t \left\| {\bf G}^{n+1} \right\|^2 ,
\label{eq:dicrete_energy_stability}
\end{align}
where
\begin{align*}
\mathcal{E}^{n+1} = 
\frac{K}{2} \left\| \nabla {\bf Q}^{n+1} \right\|^2 
+\frac{S_{Q}}{2}  \left\| {\bf Q}^{n+1} \right\|^2 
+\frac{1}{2} \left\| {\bf u}^{n+1}\right\|^2 
+ \frac{\delta t^2}{2} \left\| \nabla p^{n+1} \right\|^2
+  |r^{n+1}|^2
- C_0.
\end{align*}
\end{theorem}
\begin{proof}
Taking the $L^2$ inner products of Eq.\,\eqref{eq:first_order_scheme_Q} with ${\bf G}^{n+1}$, 
Eq.\,\eqref{eq:first_order_scheme_G} with $\dfrac{{\bf Q}^{n+1} - {\bf Q}^{n}}{\delta t}$, 
we obtain
\begin{align*}
\left( \frac{{\bf Q}^{n+1} - {\bf Q}^{n}}{\delta t}, {\bf G}^{n+1} \right) 
	+ \left( \tilde{{\bf u}}^{n+1} \cdot \nabla {\bf Q}^{n}, {\bf G}^{n+1} \right)
	- \left( S(\nabla \tilde{{\bf u}}^{n+1}, {\bf Q}^{n}), {\bf G}^{n+1} \right) 
	= M \left\| {\bf G}^{n+1}\right\|^2
\end{align*}
\begin{align*}
\left( {\bf G}^{n+1}, \frac{{\bf Q}^{n+1} - {\bf Q}^{n}}{\delta t} \right) 
	= K \left( \Delta {\bf Q}^{n+1}, \frac{{\bf Q}^{n+1} - {\bf Q}^{n}}{\delta t} \right) 
	- S_Q \left( {\bf Q}^{n+1}, \frac{{\bf Q}^{n+1} - {\bf Q}^{n}}{\delta t} \right) 
	-  r^{n+1} \left( {\bf V}^n, \frac{{\bf Q}^{n+1} - {\bf Q}^{n}}{\delta t} \right)
\end{align*}
Then
\begin{align}\label{eq:energy_Q_H}
&\frac{K}{2\delta t}\left(
	\left\| \nabla {\bf Q}^{n+1}\right\|^2 
		- \left\| \nabla {\bf Q}^{n}\right\|^2
		+ \left\| \nabla {\bf Q}^{n+1} - \nabla {\bf Q}^{n} \right\|^2
	\right)
	+ \frac{S_Q}{2\delta t}\left(
	\left\| {\bf Q}^{n+1}\right\|^2 
		- \left\| {\bf Q}^{n}\right\|^2
		+ \left\| {\bf Q}^{n+1} - {\bf Q}^{n} \right\|^2
	\right)
	+ M \left\| {\bf G}^{n+1}\right\|^2
\notag \\ = &
	\left( \tilde{{\bf u}}^{n+1} \cdot \nabla {\bf Q}^{n}, {\bf G}^{n+1} \right)
	- \left( S(\nabla \tilde{{\bf u}}^{n+1}, {\bf Q}^{n}), {\bf G}^{n+1} \right)
	-  r^{n+1} \left( {\bf V}^n, \frac{{\bf Q}^{n+1} - {\bf Q}^{n}}{\delta t} \right)
\end{align}
Taking the $L^2$ inner products of Eq.\,\eqref{eq:first_order_scheme_uhat} with $\tilde{\bf u}^{n+1}$
leads to
\begin{align*}
\left( \frac{\tilde{\bf u}^{n+1} - {\bf u}^{n}}{\delta t},\tilde{\bf u}^{n+1} \right)  
	+ \left(({\bf u}^n \cdot \nabla )\tilde{\bf u}^{n+1}, \tilde{\bf u}^{n+1} \right)
	= &
		-\eta \left\|\nabla \tilde{\bf u}^{n+1}\right\|^2 
		- \left( \nabla p^n, \tilde{\bf u}^{n+1} \right)
	\\ & 
		+ \left( \nabla \cdot \sigma({\bf Q}^n, {\bf G}^{n+1}), \tilde{\bf u}^{n+1} \right)
		- \left( \nabla {\bf Q}^n : {\bf G}^{n+1}, \tilde{\bf u}^{n+1} \right)
\end{align*}
Note Eq.\,\eqref{eq:first_order_scheme_u_p}, we have
\begin{align}\label{eq:energy_u_p_1}
{\bf u}^{n+1} + \delta t \nabla p^{n+1}  = \tilde{\bf u}^{n+1} + \delta t\nabla p^{n} 
\end{align}
Squaring both sides of Eq.\,\eqref{eq:energy_u_p_1} 
	and noticing that 
		$(\nabla p^{n+1}, {\bf u}^{n+1}) = -(p^{n+1}, \nabla \cdot {\bf u}^{n+1}) = 0$,
	we have
\begin{align*}
\left\| {\bf u}^{n+1} \right\|^2 + (\delta t)^2 \left\| \nabla p^{n+1} \right\|^2
	= \left\| \tilde{\bf u}^{n+1} \right\|^2 + (\delta t)^2 \left\| \nabla p^{n} \right\|^2
		+ 2\delta t\left( \tilde{\bf u}^{n+1}, \nabla p^{n} \right)
\end{align*}
Thus
\begin{align}\label{eq:energy_u}
&\frac{1}{2\delta t}\left( 
		\left\|{\bf u}^{n+1} \right\|^2 - \left\|{\bf u}^{n} \right\|^2
		+ \left\|\tilde{\bf u}^{n+1}-{\bf u}^{n} \right\|^2 
		\right)
	+ \frac{\delta t}{2} \left\| \nabla p^{n+1} \right\|^2
	- \frac{\delta t}{2} \left\| \nabla p^{n} \right\|^2
	+ \eta \left\|\nabla \tilde{\bf u}^{n+1}\right\|^2 
\notag \\ & \qquad =
	\left( \nabla \cdot \sigma({\bf Q}^n, {\bf G}^{n+1}), \tilde{\bf u}^{n+1} \right)
		- \left( \nabla {\bf Q}^n : {\bf G}^{n+1}, \tilde{\bf u}^{n+1} \right)
\end{align}
Taking the product of Eq.\,\eqref{eq:first_order_scheme_r} with $2r^{n+1}$ gives
\begin{align}\label{eq:energy_r}
\frac{1}{\delta t} \left( |r^{n+1}|^2 - |r^n|^2 + |r^{n+1} - r^n|^2 \right) 
	= r^{n+1} \left( {\bf V}^n, \frac{{\bf Q}^{n+1} - {\bf Q}^{n}}{\delta t} \right)
\end{align}
Combining Eqs.\,\eqref{eq:energy_Q_H}, \eqref{eq:energy_u} and \eqref{eq:energy_r}, we have
\begin{align*}
&\frac{K}{2\delta t}\left(
	\left\| \nabla {\bf Q}^{n+1}\right\|^2 
		- \left\| \nabla {\bf Q}^{n}\right\|^2
		+ \left\| \nabla {\bf Q}^{n+1} - \nabla {\bf Q}^{n} \right\|^2
	\right)
	+ \frac{S_Q}{2\delta t}\left(
	\left\| {\bf Q}^{n+1}\right\|^2 
		- \left\| {\bf Q}^{n}\right\|^2
		+ \left\| {\bf Q}^{n+1} - {\bf Q}^{n} \right\|^2
	\right)
	+ M \left\| {\bf G}^{n+1}\right\|^2
\\ & 
+ \frac{1}{2\delta t}\left( 
		\left\|{\bf u}^{n+1} \right\|^2 - \left\|{\bf u}^{n} \right\|^2
		+ \left\|\tilde{\bf u}^{n+1}-{\bf u}^{n} \right\|^2 
		\right)
	+ \frac{\delta t}{2} \left\| \nabla p^{n+1} \right\|^2
	- \frac{\delta t}{2} \left\| \nabla p^{n} \right\|^2
	+ \eta \left\|\nabla \tilde{\bf u}^{n+1}\right\|^2 
\\ &
	+\frac{1}{\delta t} \left( |r^{n+1}|^2 - |r^n|^2 + |r^{n+1} - r^n|^2 \right)
\\ = &
	\left( \tilde{{\bf u}}^{n+1} \cdot \nabla {\bf Q}^{n}, {\bf G}^{n+1} \right)
	- \left( S(\nabla \tilde{{\bf u}}^{n+1}, {\bf Q}^{n}), {\bf G}^{n+1} \right)
	+ \left( \nabla \cdot \sigma({\bf Q}^n, {\bf G}^{n+1}), \tilde{\bf u}^{n+1} \right)
		- \left( \nabla {\bf Q}^n : {\bf G}^{n+1}, \tilde{\bf u}^{n+1} \right)
	= 0.
\end{align*}
This concludes the proof.
\end{proof}

\subsection{Unique solvability}
The Eqs.\,\eqref{eq:first_order_scheme_Q}-\eqref{eq:first_order_scheme_uhat} 
show the scheme is linear with respect to unknowns ${\bf Q}^{n+1}$ and $\tilde{\bf u}^{n+1}$.
We shall prove that the solution of this system is unique. 
\begin{theorem}
The scheme defined in \eqref{eq:first_order_scheme_Q}-\eqref{eq:first_order_scheme_uhat} 
admits a unique solution.
\end{theorem}
\begin{proof}
Assume
that for given ${\bf Q}^{n}$, ${\bf Q}^{n-1}$, $\tilde{\bf u}^{n}$ and $\tilde{\bf u}^{n-1}$,
the system has two solutions $\left( {\bf Q}_1, \tilde{\bf u}_1 \right)$ 
	and $\left( {\bf Q}_2, \tilde{\bf u}_2 \right)$.
Let ${\bf Q}_0 = {\bf Q}_1 - {\bf Q}_2$, 
	$\tilde{\bf u}_0 = \tilde{\bf u}_1  - \tilde{\bf u}_2$,
	we have
\begin{align}
&
\frac{{\bf Q}_0}{\delta t}
	+ \tilde{{\bf u}}_0 \cdot \nabla {\bf Q}^{n}
	- S(\nabla \tilde{{\bf u}}_0, {\bf Q}^{n}) = M {\bf G}_0,
	\label{eq:unique_Q}  
\\    &
\frac{\tilde{\bf u}_0}{\delta t} + ({\bf u}^n \cdot \nabla )\tilde{\bf u}_0
	= \eta \Delta \tilde{\bf u}_0  + \nabla \cdot \sigma({\bf Q}^n, {\bf G}_0)
	- \nabla {\bf Q}^n : {\bf G}_0,
	\label{eq:unique_u} 
\end{align}
where
\begin{align}
&{\bf G}_0 = K\Delta {\bf Q}_0  -S_Q {\bf Q}_0 - r_0 {\bf V}^n, 
\label{eq:unique_G} 
\\ &
r_0 = \frac{1}{2}\left( {\bf V}^n, {\bf Q}_0 \right). 
\label{eq:unique_r} 
\end{align}
Taking the $L^2$-inner products of Eq.\,\eqref{eq:unique_Q} with ${\bf G}^0$, 
	Eq.\,\eqref{eq:unique_G} with ${\bf Q}^0$ and 
	Eq.\,\eqref{eq:unique_u} with $\tilde{\bf u}^0$,
	and using the relationship \eqref{eq:unique_r}, we can obtain
\begin{align*}
&
\frac{1}{\delta t}\left( {\bf Q}_0, {\bf G}_0 \right)
	+ \left( \tilde{{\bf u}}_0 \cdot \nabla {\bf Q}^{n}, {\bf G}_0 \right)
	- \left( S(\nabla \tilde{{\bf u}}_0, {\bf Q}^{n}), {\bf G}_0 \right) 
	= M \left\| {\bf G}_0\right\|^2,
\\ &
\left( {\bf G}_0, {\bf Q}_0\right) = 
	-K \left\| \nabla{\bf Q}_0\right\|^2  -S_Q \left\| {\bf Q}_0\right\|^2  
	- 2|r_0|^2,
\\ &
\frac{1}{\delta t}\left\| \tilde{\bf u}_0 \right\|^2 
	= -\eta \left\| \nabla \tilde{\bf u}_0\right\|^2  
		+ \left( \nabla \cdot \sigma({\bf Q}^n, {\bf G}_0), \tilde{\bf u}_0 \right)
		- \left( \nabla {\bf Q}^n : {\bf G}_0, \tilde{\bf u}_0 \right).
\end{align*}
The three equations above yield
\begin{align*}
\frac{K}{\delta t} \left\| \nabla{\bf Q}_0\right\|^2  +\frac{S_Q}{\delta t} \left\| {\bf Q}_0\right\|^2  
	+ \frac{2}{\delta t}|r_0|^2 
	+ \frac{1}{\delta t}\left\| \tilde{\bf u}_0 \right\|^2 
+ 
	+ M \left\| {\bf G}_0\right\|^2 + \eta \left\| \nabla \tilde{\bf u}_0\right\|^2 
	= 0,
\end{align*}
so that ${\bf Q}_0 = 0$, $\tilde{\bf u}_0 = 0$.
\end{proof}

\section{Error estimate}
\label{sec:Error_estimate}
In this section, 
	we aim to give error estimates for the scheme\eqref{eq:first_order_scheme_Q}-\eqref{eq:first_order_scheme_nabla_u}. 
Our final results of the error estimates for phase function ${\bf Q}$, chemical potential ${\bf G}$, the velocity {\bf u} are stated in the Theorem\ref{theorem:error_estimates}.
The key point of error estimate lies in the boundness of $\left\| {\bf Q}^n \right\|_{L^{\infty}}$,
which simplifies the estimates of some complex nonlinear terms.
Here, we couple the proof of $L^{\infty}$ boundedness and error estimates, 
and use mathematical induction to derive the results.
For the sake of simplicity, 
we use the second boundary condition in our proof process, i.e., 
\begin{align}\label{boundary:dirichlet}
{\bf u}^{n+1}\cdot \mathbf{n} = 0,
\tilde{\bf u}^{n+1}|_{\partial \Omega} = 0,
{\bf Q}^{n+1}|_{\partial \Omega} = {\bf Q}^0 
\ or \
\partial_n {\bf Q}^{n+1}|_{\partial \Omega} = 0. 
\end{align} 
In fact, our conclusion is true for the periodic boundary.

We denote that
\begin{align*}
e_Q^n = {\bf Q}^n - {\bf Q}(t^n), 
	\quad 
	e_G^n = {\bf G}^n - {\bf G}(t^n),
	\quad
	e_r^n = r^n - r(t^n),
\\ 
e_u^n = {\bf u}^n - {\bf u}(t^n),
	\quad
	\tilde{e}_u^n = \tilde{\bf u}^n - {\bf u}(t^n), 
	\quad
	e_p^n = p^n - p(t^n).
\end{align*}
Theorem \ref{theorem:discrete_energy_stable} imply that
	the functions ${\bf Q}^n$, $\tilde{\bf u}^n$ and $r^n$ are bounded 
	as follows:
	\begin{align*}
		\left\| \tilde{\bf u}^n \right\|, 
		\left\| {\bf Q}^n \right\|_{H^1}, 
		|r^n| \leq C, \quad 0 \leq n \leq N,
	\end{align*}
	where $N = T/\delta t$ and $C$ depends on $\Omega$ and the initial conditions.

Supposing that the exact solutions ${\bf Q}$, ${\bf u}$, $p$ and $r$ 
for the system possess the following regularity:
\begin{equation}\label{assumptions_for_exact_solution}
\begin{aligned}
&
	{\bf Q} \in L^{\infty}(0,T;H^2)\cap L^{\infty}(0,T;H^3), 
	\qquad
	{\bf u} \in L^{\infty}(0,T;H^2),
\\ &
	{\bf Q}_t \in L^{\infty}(0,T;H^2),
	\qquad
	{\bf u}_t \in L^{\infty}(0,T;H^1),
	\qquad
	p_t \in L^{\infty}(0,T;H^1),
\\ &
	{\bf Q}_{tt} \in L^{\infty}(0,T;L^2),
	\qquad
	{\bf u}_{tt} \in L^{\infty}(0,T;L^2),
	\qquad
	r_{tt} \in L^{2}(0,T).
\\
	\end{aligned}
\end{equation}
The regularity hypotheses for the scalar auxiliary variable $r(t)$ are not necessary. Instead,
we can obtain through the regularity of the exact solutions \eqref{assumptions_for_exact_solution}.
\begin{lemma}
Under the regularity hypotheses\eqref{assumptions_for_exact_solution}, 
	we have $r_{tt} \in L^2(0,T)$.
\end{lemma}

\begin{lemma}\label{lemma:infty}
Assuming the assumption \eqref{assumptions_for_exact_solution} holds,
then there exists positive constant $\delta t_0$, and for any $\delta t < \delta t_0$, the
solutions ${\bf Q}^n$ satisfy
\begin{align*}
\left\| {\bf Q}^n \right\|_{L^{\infty}} \leq \mathcal{C} 
	= \mathop{max}\limits_{0\leq t \leq T} \left\| {\bf Q}(t) \right\|_{L^{\infty}} + 2.
\end{align*}
where $n = 0,1,2,...,N$.
\end{lemma}
\begin{proof}
We prove this lemma with mathematical induction.
When $n = 0$, we choose ${\bf Q}^0 = {\bf Q}(t^0)$. 
Supposing that $\left\| {\bf Q}^n \right\|_{L^{\infty}} \leq \mathcal{C}$ is valid for $n = 0,1,2,...,k$, 
we prove $\left\| {\bf Q}^{k+1} \right\|_{L^{\infty}} \leq \mathcal{C}$ is also valid 
in 5 Steps.


\textbf{Step 1.}
Subtracting Eqs.\,\eqref{eq:equal_model} from Eqs.\,\eqref{eq:first_order_scheme_Q}-\eqref{eq:first_order_scheme_u_p},
we can get the error equations as follows.
\begin{align}
&\frac{e_{Q}^{n+1} - e_{ Q}^{n}}{\delta t} 
	+ \tilde{{\bf u}}^{n+1} \cdot \nabla {\bf Q}^{n} - {\bf u}(t^{n+1}) \cdot \nabla {\bf Q}(t^{n})
\notag \\ & \qquad \qquad \qquad
	- S(\nabla \tilde{{\bf u}}^{n+1}, {\bf Q}^{n}) + S(\nabla{\bf u}(t^{n+1}), {\bf Q}(t^{n})) 
	= M e_{G}^{n+1} + R_Q^{n+1},
\label{eq:error_first_order_scheme_Q}
\\ &
e_{G}^{n+1} = K \Delta e_{Q}^{n+1} - S_Q e_{Q}^{n+1} -  r^{n+1}{\bf V}^n + r(t^{n+1}){\bf V}(t^{n}) + R_G^{n+1},
\label{eq:error_first_order_scheme_G}
\\ &
\frac{\tilde{e}_u^{n+1} - {e}_u^{n}}{\delta t} 
	+ ({\bf u}^n \cdot \nabla )\tilde{\bf u}^{n+1} - ({\bf u}(t^{n}) \cdot \nabla ){\bf u}(t^{n+1})
\notag \\ & \qquad \qquad \qquad
	= \eta \Delta \tilde{e}_u^{n+1} - \nabla p^n + \nabla p(t^{n+1}) 
		+ \nabla \cdot \sigma({\bf Q}^n, {\bf G}^{n+1}) - \nabla \cdot \sigma({\bf Q}(t^{n}), {\bf G}(t^{n+1}))
\notag \\ & \qquad \qquad \qquad
	- \nabla {\bf Q}^n : {\bf G}^{n+1} + \nabla {\bf Q}(t^{n}) : {\bf G}(t^{n+1}) + R_{u}^{n+1},
\label{eq:error_first_order_scheme_uhat}
\\ &
\frac{e_{ u}^{n+1} - \tilde{e}_u^{n+1} }{\delta t} + \nabla (p^{n+1} - p^n) = 0,
\label{eq:error_first_order_scheme_u_p}
\\ &
\frac{e_r^{n+1} - e_r^n}{\delta t} = \frac{1}{2}\int_{\Omega} \left(
		{\bf V}^n :\frac{{\bf Q}^{n+1} - {\bf Q}^{n}}{\delta t} 
									- {\bf V}(t^{n}):\frac{{\bf Q}(t^{n+1})- {\bf Q}(t^{n})}{\delta t} 
								\right) dx
		+  R_r^{n+1}.
\label{eq:error_first_order_scheme_r}
\end{align}
where
\begin{align*}
R_Q^{n+1} &= -{\bf u}(t^{n+1}) \cdot \nabla {\bf Q}(t^{n}) + {\bf u}(t^{n+1}) \cdot \nabla {\bf Q}(t^{n+1}) 
	+ S(\nabla{\bf u}(t^{n+1}), {\bf Q}(t^{n})) - S(\nabla{\bf u}(t^{n+1}), {\bf Q}(t^{n+1})) 
	\\ &
	- \frac{{\bf Q}(t^{n+1}) - {\bf Q}(t^{n})}{\delta t} + {\bf Q}_t(t^{n+1}) , 
\\
R_G^{n+1} &= r(t^{n+1}){\bf V}(t^{n+1}) - r(t^{n+1}){\bf V}(t^{n}), 
\\ 
R_u^{n+1} & 
	= R_{u1}^{n+1} + \nabla \cdot R_{u2}^{n+1},
\\
R_{u1}^{n+1} & = 
	- \frac{{\bf u}(t^{n+1}) - {\bf u}(t^{n})}{\delta t} + {\bf u}_t(t^{n+1})
	- ({\bf u}(t^{n}) \cdot \nabla ){\bf u}(t^{n+1}) 
		+ ({\bf u}(t^{n+1}) \cdot \nabla ){\bf u}(t^{n+1})
\\ &
	- \nabla {\bf Q}(t^{n}) : {\bf G}(t^{n+1})
	+ \nabla {\bf Q}(t^{n+1}) : {\bf G}(t^{n+1}), 
\\ 
R_{u2}^{n+1} & = 
	\nabla \cdot \sigma({\bf Q}(t^{n}), {\bf G}(t^{n+1}))
	- \nabla \cdot \sigma({\bf Q}(t^{n+1}), {\bf G}(t^{n+1}))
\\
R_r^{n+1} &= \frac{1}{2}\int_{\Omega} \left(
	{\bf V}(t^{n}):\frac{{\bf Q}(t^{n+1})- {\bf Q}(t^{n})}{\delta t} 
	- {\bf V}(t^{n+1}):{\bf Q}_t (t^{n+1})
	\right) dx	
			- \frac{{r(t^{n+1}) - r(t^{n})}}{\delta t} + r_t(t^{n+1}).	
\end{align*}

\begin{lemma}\label{lemma:R_Q_G_u}
Under assumption \eqref{assumptions_for_exact_solution}, there holds
\begin{align*}
\left\| R_Q^{n+1} \right\| +
		\left\| R_G^{n+1} \right\|_{H^1} +
		\left\| R_{u1}^{n+1} \right\| +
		\left\| R_{u2}^{n+1} \right\| \leq C \delta t,
	\qquad 
		| R_r^{n+1} | \leq C 
			| r_{tt} |_{L^2(t^n,t^{n+1})}\delta t^{\frac{1}{2}}.
\end{align*}
\end{lemma}
\begin{lemma}\label{lemma:Vn_Ve}
Combining assumption \eqref{assumptions_for_exact_solution}, 
Theorem \ref{theorem:discrete_energy_stable}, 
we have
\begin{align*}
	& \left\| {\bf V}^n - {\bf V}(t^n) \right\| \leq C \left\|  e_{Q}^n  \right\|, \\
	& \left\| \nabla {\bf V}^n - \nabla {\bf V}(t^n) \right\| \leq C 
	\left\|  \nabla e_{Q}^n  \right\| + C \left\| e_{Q}^n  \right\|.
\end{align*}
\end{lemma}
\begin{proof}
We take the estimate of $\left\| {\bf V}^n - {\bf V}(t^n) \right\|$ as an example. 
The estimate of $\left\| \nabla {\bf V}^n - \nabla {\bf V}(t^n) \right\|$ is similar, 
	so we omit it.
We can rewrite ${\bf V}^n - {\bf V}(t^n)$ as follows
\begin{align*}
 {\bf V}^n - {\bf V}(t^n)
 	= &g(Q(t^n))\frac{\mathcal{E}_1({\bf Q}(t^n)) - \mathcal{E}_1({\bf Q}^n)}{\sqrt{\mathcal{E}_1({\bf Q}^n) } \sqrt{\mathcal{E}_1({\bf Q}(t^n))} (\sqrt{\mathcal{E}_1({\bf Q}^n) }+\sqrt{\mathcal{E}_1({\bf Q}(t^n)) })}
 	&
 	\\ & 
 		+ \frac{g({\bf Q}^n) - g({\bf Q}(t^n))}{\sqrt{\mathcal{E}_1({\bf Q}^n) }}.
 	& := \sum_{i=1}^2 \mathcal{M}_i
\end{align*}
Note 
\begin{align*}
& 
\left\| \mathcal{M}_1 \right\|
\leq C \left\| g({\bf Q}(t^n)) \right\| \left\| e_{Q}^n  \right\| 
\leq C \left\| e_{Q}^n  \right\|, 
&
\left\| \mathcal{M}_2 \right\| 
\leq C  \left\| \frac{\delta g}{\delta Q}(\xi_2) e_{Q}^n  \right\| 
\leq C  \left\|  e_{Q}^n  \right\|, 
	\quad \xi_2 \in ({\bf Q}^n, {\bf Q}(t^n)).
\end{align*}
The estimate for $\left\| {\bf V}^n - {\bf V}(t^n) \right\|$ can be derived.

\end{proof}

\textbf{Step 2.}
 We derive the error estimate formulas for error functions.
 
Taking the $L^2$ product of Eq.\,\eqref{eq:error_first_order_scheme_Q} and 
	Eq.\,\eqref{eq:error_first_order_scheme_G} 
with $e_{G}^{n+1}$ and $\frac{e_{Q}^{n+1} - e_{Q}^{n}}{\delta t} $ respectively, we have 
\begin{align*}
&\left( \frac{e_{Q}^{n+1} - e_{ Q}^{n}}{\delta t}, e_G^{n+1} \right)
	+ \underset{I_1}{\uwave{\left( \tilde{{\bf u}}^{n+1} \cdot \nabla {\bf Q}^{n} - {\bf u}(t^{n+1}) \cdot \nabla {\bf Q}(t^{n}), e_G^{n+1} \right)}}
\notag \\ & \qquad \qquad \qquad
	+ \underset{I_2}{\uwave{\left( - S(\nabla \tilde{{\bf u}}^{n+1}, {\bf Q}^{n}) + S(\nabla{\bf u}(t^{n+1}), {\bf Q}(t^{n})), e_G^{n+1} \right)}}
	= M \left\| e_{G}^{n+1} \right\|^2 + \underset{-I_3}{\uwave{(R_Q^{n+1}, e_G^{n+1})}},
\\ &
\left( e_G^{n+1} , \frac{e_{Q}^{n+1} - e_{ Q}^{n}}{\delta t} \right)
	= 
	K\left(\Delta e_Q^{n+1} ,\frac{e_{Q}^{n+1} - e_{ Q}^{n}}{\delta t} \right)
	- S_Q \left(e_Q^{n+1} ,\frac{e_{Q}^{n+1} - e_{ Q}^{n}}{\delta t} \right)
\\ &\qquad \qquad \qquad \qquad \qquad
	 + \underset{I_4}{\uwave{\left(-  r^{n+1}{\bf V}^n + r(t^{n+1}){\bf V}(t^{n}), \frac{e_{Q}^{n+1} - e_{ Q}^{n}}{\delta t}\right)}}
	 + \underset{I_5}{\uwave{\left(R_G^{n+1}, \frac{e_{Q}^{n+1} - e_{ Q}^{n}}{\delta t}\right)}}.
\end{align*}
Taking the $L^2$ product of Eq.\,\eqref{eq:error_first_order_scheme_uhat} 
with $\tilde{e}_{u}^{n+1}$, we have 
\begin{align*}
&
\left(\frac{\tilde{e}_u^{n+1} - {e}_u^{n}}{\delta t}, \tilde{e}_{u}^{n+1}\right)
	+ \underset{J_1}{\uwave{\left(({\bf u}^n \cdot \nabla )\tilde{\bf u}^{n+1} - ({\bf u}(t^{n}) \cdot \nabla ){\bf u}(t^{n+1}), 
		\tilde{e}_{u}^{n+1}\right)}}
\notag \\ & \qquad \qquad \qquad
	= 
	- \eta \left\| \nabla \tilde{e}_u^{n+1}\right\|^2 
	+ \underset{J_2}{\uwave{\left( - \nabla p^n + \nabla p(t^{n+1}), \tilde{e}_{u}^{n+1}\right)}}
		+ 
		\underset{J_3}{\uwave{\left( \nabla \cdot \sigma({\bf Q}^n, {\bf G}^{n+1}) - \nabla \cdot \sigma({\bf Q}(t^{n}), {\bf G}(t^{n+1})), \tilde{e}_{u}^{n+1}\right)}}
\notag \\ & \qquad \qquad \qquad
	+\underset{J_4}{\uwave{\left(- \nabla {\bf Q}^n : {\bf G}^{n+1} + \nabla {\bf Q}(t^{n}) : {\bf G}(t^{n+1})
		, \tilde{e}_{u}^{n+1}\right)}} 
	+ \underset{J_5}{\uwave{\left( R_u^{n+1}, \tilde{e}_{u}^{n+1}\right)}},
\end{align*}
Taking the $L^2$ product of Eq.\,\eqref{eq:error_first_order_scheme_r} 
with $2{e}_{r}^{n+1}$, we have 
\begin{align*}
\frac{1}{\delta t}\left(|{e}_{r}^{n+1}|^2 - |{e}_{r}^{n}|^2 + |{e}_{r}^{n+1} - e_r^{n}|^2 \right)
	=
		e_r^{n+1} \left( {\bf V}^n,  \frac{e_Q^{n+1} - e_Q^n}{\delta t}\right)
		+ \underset{K_1}{\uwave{e_r^{n+1} \left( {\bf V}^n - {\bf V}(t^n), \frac{{\bf Q}(t^{n+1}) - {\bf Q}(t^{n})}{\delta t} \right)}}
\\ 
		+ \underset{K_2}{\uwave{2e_r^{n+1} R_r^{n+1}}}.
\end{align*}
Thus, we have
\begin{align}\label{eq:error_formulas}
&\frac{K}{2\delta t}\left( 
	\left\| \nabla e_Q^{n+1} \right\|^2 - \left\| \nabla e_Q^{n} \right\|^2
	+ \left\| \nabla e_Q^{n+1} - \nabla e_Q^{n} \right\|^2
\right)
+ \frac{S_Q}{2\delta t}\left( 
	\left\| e_Q^{n+1} \right\|^2 - \left\| e_Q^{n} \right\|^2
	+ \left\| e_Q^{n+1} - e_Q^{n} \right\|^2
\right)
+ M \left\| e_{G}^{n+1} \right\|^2 
\notag \\ &
+ \left(\frac{\tilde{e}_u^{n+1} - {e}_u^{n}}{\delta t}, \tilde{e}_{u}^{n+1}\right)
	+ \eta \left\| \nabla \tilde{e}_u^{n+1}\right\|^2
+ \frac{1}{\delta t}\left(|{e}_{r}^{n+1}|^2 - |{e}_{r}^{n}|^2 + |{e}_{r}^{n+1} - e_r^{n}|^2 \right)
\notag \\ &
	= e_r^{n+1} \left( {\bf V}^n,  \frac{e_Q^{n+1} - e_Q^n}{\delta t}\right)
	+ \sum_{i=1}^5 I_i 
	+ \sum_{i=1}^5 J_i
	+ \sum_{i=1}^2 K_i.
\end{align}

\textbf{Step 3.}
We estimate each term at the right of Eq.\,\eqref{eq:error_formulas}.
 
Using assumption \ref{assumptions_for_exact_solution},
	H$\ddot{o}$lder's inequality, 
	the Sobolev inequalities, 
	Lemma \ref{Lemma: 1_Gag_Niren}, Theorem \ref{theorem:discrete_energy_stable}
	and the Cauchy inequality,
	we have 
\begin{align*}
I_1 & = \left( \tilde{{\bf u}}^{n+1} \cdot \nabla {\bf Q}^{n} - {\bf u}(t^{n+1}) \cdot \nabla {\bf Q}(t^{n}), e_G^{n+1} \right)
\\ & = 
	\left( \tilde{{e}}_u^{n+1} \cdot \nabla {\bf Q}^{n} , e_G^{n+1} \right)
		+ \left( {\bf u}(t^{n+1}) \cdot \nabla e_{\bf Q}^n, e_G^{n+1} \right)
\\ & \leq
	\left( \tilde{{e}}_u^{n+1} \cdot \nabla {\bf Q}^{n} , e_G^{n+1} \right)
		+ \left\|{\bf u}(t^{n+1}) \right\|_{L^{\infty}}  \left\|\nabla e_Q^{n} \right\| \left\|e_G^{n+1} \right\| 
\\ & \leq 
	\underset{I_{11}}{\uwave{\left( \tilde{{e}}_u^{n+1} \cdot \nabla {\bf Q}^{n} , e_G^{n+1} \right)}}
		+ \frac{M}{10} \left\|e_G^{n+1} \right\|^2 + C \left\|\nabla e_Q^{n} \right\|^2.
\end{align*}
Note
\begin{align*}
&
- S(\nabla \tilde{{\bf u}}^{n+1}, {\bf Q}^{n}) + S(\nabla{\bf u}(t^{n+1}), {\bf Q}(t^{n}))
\\
	= & 
	- S(\nabla \tilde{{\bf u}}^{n+1}, {\bf Q}^{n}) 
		+ S(\nabla{\bf u}(t^{n+1}), {\bf Q}^{n})
		- S(\nabla{\bf u}(t^{n+1}), {\bf Q}^{n})
	+ S(\nabla{\bf u}(t^{n+1}), {\bf Q}(t^{n}))
\\ 
	= & 
		- S(\nabla \tilde{e}_u^{n+1}, {\bf Q}^{n}) 
		-S(\nabla{\bf u}(t^{n+1}), e_{\bf Q}^{n})
		- 2a \left(  e_{\bf Q}^{n}:D(t^{n+1}) \right)e_{\bf Q}^{n}
\\ &
		+ 2a \left( {\bf Q}^{n}:D(t^{n+1}) \right){\bf Q}^{n}
		- 2a \left( {\bf Q}(t^{n}):D(t^{n+1}) \right){\bf Q}(t^{n}),
\end{align*}
we can transform and estimate $I_2$ as follows
\begin{align*}
I_2 =  &- \left( S(\nabla \tilde{e}_u^{n+1}, {\bf Q}^{n}) , e_G^{n+1} \right)
		 -\left( S(\nabla{\bf u}(t^{n+1}), e_{\bf Q}^{n}), e_G^{n+1} \right)
		- 2a \left( 
			\left( e_{\bf Q}^{n}:D(t^{n+1}) \right)e_{\bf Q}^{n}, e_G^{n+1} \right)
	\\ &
		+ 2a \left( \left( {\bf Q}^{n}:D(t^{n+1}) \right){\bf Q}^{n} , e_G^{n+1} \right)
		- 2a \left( \left( {\bf Q}(t^{n}):D(t^{n+1}) \right){\bf Q}(t^{n}) , e_G^{n+1} \right)
	\\ \leq & 
	- \left( S(\nabla \tilde{e}_u^{n+1}, {\bf Q}^{n}) , e_G^{n+1} \right)
	+ C \left\| \nabla {\bf u}(t^{n+1}) \right\|_{L^4} \left\| e_Q^n \right\|_{L^4} \left\|e_G^{n+1} \right\| 
	\\ & 
	+ 2a\left(  
		e_{\bf Q}^{n}:{\bf D}(t^{n+1}){\bf Q}^n
		+ {\bf Q}(t^{n}):{\bf D}(t^{n+1})e_{\bf Q}^{n}
		, e_G^{n+1} \right)
\\ \leq & 
	-\left( S(\nabla \tilde{e}_u^{n+1}, {\bf Q}^{n}) , e_G^{n+1} \right)
	+ \left\| \nabla {\bf u}(t^{n+1}) \right\|_{L^4} \left\| e_Q^n \right\|_{L^4} \left\|e_G^{n+1} \right\| 
	\\ & 
	+ C \left\| e_{\bf Q}^{n}\right\|_{L^6}\left\|{\bf D}(t^{n+1})\right\|_{L^{6}}\left\| {\bf Q}^n \right\|_{L^6}
		\left\| e_G^{n+1}  \right\|
		+ \left\|{\bf Q}(t^{n})\right\|_{L^{6}}
			\left\|{\bf D}(t^{n+1})\right\|_{L^{6}}
			\left\| e_{\bf Q}^{n} \right\|_{L^6}
		\left\| e_G^{n+1}  \right\| 
	\\ \leq &
	-\underset{I_{21}}{\uwave{ \left( S(\nabla \tilde{e}_u^{n+1}, {\bf Q}^{n}) , e_G^{n+1} \right)}}
	+ \frac{M}{10} \left\|e_G^{n+1} \right\|^2 + C \left\| \nabla e_Q^n \right\|^2
	+ C \left\| e_Q^n \right\|^2.
\end{align*}
Meanwhile, $I_4$ can be estimated as follows
\begin{align*}
&\left(-  r^{n+1}{\bf V}^n + r(t^{n+1}){\bf V}(t^{n}), \frac{e_{Q}^{n+1} - e_{ Q}^{n}}{\delta t}\right)
\\
	= & 
	- e_r^{n+1}\left({\bf V}^n , \frac{e_{Q}^{n+1} - e_{ Q}^{n}}{\delta t}\right)
\\ & 
	+ r(t^{n+1})\left(- {\bf V}^n + {\bf V}(t^{n}), \right.
	\\ & \qquad \qquad \left.
			- \tilde{{\bf u}}^{n+1} \cdot \nabla {\bf Q}^{n} + {\bf u}(t^{n+1}) \cdot \nabla {\bf Q}(t^{n})
			+ S(\nabla \tilde{{\bf u}}^{n+1}, {\bf Q}^{n}) - S(\nabla{\bf u}(t^{n+1}), {\bf Q}(t^{n})) 
			+ M e_{G}^{n+1} + R_Q^{n+1}\right)
\\ \leq &
	- e_r^{n+1}\left({\bf V}^n , \frac{e_{Q}^{n+1} - e_{ Q}^{n}}{\delta t}\right)
	+ C \left\|  {\bf V}^n - {\bf V}(t^{n})  \right\|^2 + \frac{M}{10} \left\| e_G^{n+1}\right\|^2
	+ C \left\| R_Q^{n+1}\right\|^2
\\ &
	+ r(t^{n+1}) \left( -\tilde{\bf u}^{n+1} \cdot \nabla e_{\bf Q}^n 
		+ \tilde{e}_u^{n+1} \cdot \nabla {\bf Q}(t^n) ,
		-{\bf V}^n + {\bf V}(t^n)
		\right)
\\ &
	+ r(t^{n+1}) \left( S(\nabla \tilde{{\bf u}}^{n+1}, {\bf Q}^{n}) - S(\nabla{\bf u}(t^{n+1}), {\bf Q}(t^{n})) ,
		-{\bf V}^n + {\bf V}(t^n)
		\right)
\\ \leq &
	- e_r^{n+1}\left({\bf V}^n , \frac{e_{Q}^{n+1} - e_{ Q}^{n}}{\delta t}\right)
	+ \frac{\eta}{12}\left\| \nabla \tilde{e}_u^{n+1}\right\|^2
	+ \frac{M}{10} \left\| e_G^{n+1}\right\|^2
	+ C(1+\left\| \nabla \tilde{\bf u}^{n+1} \right\|^2)\left\| \nabla e_Q^n \right\|^2
	+ C\left\| e_Q^n \right\|^2
	+ C \left\| R_Q^{n+1}\right\|^2,
\end{align*}
due to
\begin{align*}
&
r(t^{n+1}) \left( -\tilde{\bf u}^{n+1} \cdot \nabla e_{\bf Q}^n 
		+ \tilde{e}_u^{n+1} \cdot \nabla {\bf Q}(t^n) ,
		-{\bf V}^n + {\bf V}(t^n)
		\right)
\\ \leq & C
		\left\|\tilde{\bf u}^{n+1} \right\|_{L^4}
		\left\|  \nabla e_{\bf Q}^n \right\|
		\left\| -{\bf V}^n + {\bf V}(t^n) \right\|_{L^4}
	+ C
	\left\| -{\bf V}^n + {\bf V}(t^n) \right\|
		\left\| \tilde{e}_u^{n+1} \right\|_{L^4}
		\left\|  \nabla {\bf Q}(t^n) \right\|_{L^4}
\\ \leq & 
	\frac{\eta}{24} \left\|  \nabla \tilde{e}_u^{n+1} \right\|^2
	+ C \left\| \tilde{\bf u}^{n+1} \right\|^2_{H^1} \left\| \nabla e_{\bf Q}^n \right\|^2
	+ C\left\| {\bf V}^n - {\bf V}(t^n) \right\|^2_{H^1}
\\ \leq &
	\frac{\eta}{24} \left\|  \nabla \tilde{e}_u^{n+1} \right\|^2
	+ C\left\|  e_{\bf Q}^n \right\|^2
	+ C (1 + \left\| \nabla \tilde{\bf u}^{n+1} \right\|^2)  \left\|  \nabla e_{\bf Q}^n \right\|^2,
\end{align*}
and
\begin{align*}
&r(t^{n+1}) \left( S(\nabla \tilde{{\bf u}}^{n+1}, {\bf Q}^{n}) - S(\nabla{\bf u}(t^{n+1}), {\bf Q}(t^{n})) ,
		-{\bf V}^n + {\bf V}(t^n)
		\right)
\\ = & r(t^{n+1}) \left( 
		S(\nabla \tilde{e}_u^{n+1}, {\bf Q}^{n}),
		-{\bf V}^n + {\bf V}(t^n)
		\right)
\\ &
	+ r(t^{n+1}) \left( 
		S(\nabla{\bf u}(t^{n+1}), {\bf Q}^n) 
		- S(\nabla{\bf u}(t^{n+1}), {\bf Q}(t^{n})) ,
		-{\bf V}^n + {\bf V}(t^n)
		\right)
\\ \leq &
	C \left\| \nabla \tilde{e}_u^{n+1}\right\|
		\left( \left\| {\bf Q}^n \right\|_{L^6}^2 
			\left\| {\bf V}^n - {\bf V}(t^n) \right\|_{L^6}
			+
			\left\| {\bf Q}^n \right\|_{L^4} 
			\left\| {\bf V}^n - {\bf V}(t^n) \right\|_{L^4}
		\right)
\\ &
	+
	C \left\| \nabla {\bf u}(t^{n+1})\right\|_{L^6}
		\left\| e_{Q}^n \right\|
		\left\| {\bf V}^n - {\bf V}(t^n) \right\|_{L^6}
		(
		\left\| {\bf Q}^n \right\|_{L^6}
		+
		 \left\| {\bf Q}(t^n) \right\|_{L^6}
		)
\\ &
	+
	C \left\| \nabla {\bf u}(t^{n+1})\right\|_{L^4}
		\left\| e_{Q}^n \right\|
		\left\| {\bf V}^n - {\bf V}(t^n) \right\|_{L^4}
\\ \leq &
	\frac{\eta}{24}\left\| \nabla \tilde{e}_u^{n+1}\right\|^2
	+ C\left\| {\bf V}^n - {\bf V}(t^n) \right\|^2_{H^1}
	+ C\left\| e_Q^n \right\|^2
\\ \leq &
	\frac{\eta}{24}\left\| \nabla \tilde{e}_u^{n+1}\right\|^2
	+ C\left\| e_Q^n \right\|^2 + C\left\| \nabla e_Q^n \right\|^2.
\end{align*}
Similar to $I_4$, we have 
\begin{align*}
I_5 = &\left(R_G^{n+1}, \frac{e_{Q}^{n+1} - e_{ Q}^{n}}{\delta t}\right)
\\ = &
	\left(
		R_G^{n+1}, 
		- \tilde{{\bf u}}^{n+1} \cdot \nabla {\bf Q}^{n} + {\bf u}(t^{n+1}) \cdot \nabla {\bf Q}(t^{n})
			+ S(\nabla \tilde{{\bf u}}^{n+1}, {\bf Q}^{n}) - S(\nabla{\bf u}(t^{n+1}), {\bf Q}(t^{n})) 
			+ M e_{G}^{n+1} + R_Q^{n+1}
	\right)
\\ \leq &
	\frac{\eta}{12}\left\| \nabla \tilde{e}_u^{n+1}\right\|^2
	+ \frac{M}{10} \left\| e_G^{n+1}\right\|^2
	+ C(1+\left\| \nabla \tilde{\bf u}^{n+1} \right\|^2)\left\| \nabla e_Q^n \right\|^2
	+ C\left\| e_Q^n \right\|^2
	+ C \left\| R_Q^{n+1}\right\|^2 
	+ C \left\| R_G^{n+1}\right\|_{H^1}^2.
\end{align*}
Applying integration by parts and $\nabla \cdot {\bf u}^n = 0$, 
we have  
\begin{align*}
J_1 = &  \left(({\bf u}^n \cdot \nabla )\tilde{\bf u}^{n+1} - ({\bf u}(t^{n}) \cdot \nabla ){\bf u}(t^{n+1}), 
		\tilde{e}_{u}^{n+1}\right)
	\\ = & 
		\underset{J_{11}}{\uwave{
		\left(({\bf u}^n \cdot \nabla )\tilde{e}_u^{n+1} , 
		\tilde{e}_{u}^{n+1}\right)}}
		+ 
		\left((e_u^n \cdot \nabla ){\bf u}(t^{n+1}) , 
		\tilde{e}_{u}^{n+1}\right)
\\ \leq &
	C \left\| \tilde{e}_{u}^{n+1} \right\|_{L^4}
	\left\| \nabla {\bf u}(t^{n+1}) \right\|_{L^4}
		\left\| {e}_{u}^{n} \right\|
	\\ \leq &
		\frac{\eta}{12} \left\| \nabla \tilde{e}_{u}^{n+1} \right\|^2
		+ C \left\| {e}_{u}^{n} \right\|^2.
\end{align*}
Taking the $L^2$ inner product of Eq.\,\eqref{eq:error_first_order_scheme_u_p} 
	with $\frac{e_{ u}^{n+1} + \tilde{e}_u^{n+1} }{2}$, we have
\begin{align*}
\frac{\left\| e_{ u}^{n+1}\right\|^2 - \left\| \tilde{e}_u^{n+1}\right\|^2}{2\delta t}
	+\frac{1}{2}\left(\nabla(p^{n+1} - p^n), \tilde{e}_u^{n+1} \right) = 0.
\end{align*}
Therefore,
\begin{align*}
&J_2 - \frac{1}{2}\left(\nabla(p^{n+1} - p^n), \tilde{e}_u^{n+1} \right)
\\ = &  - \frac{1}{2}\left( \nabla (p^{n+1} +  p^n - 2 p(t^{n+1})), \tilde{e}_{u}^{n+1}\right)
\\ = &
	- \frac{1}{2}\left( \nabla (e_p^{n+1} +  e_p^n - p(t^{n+1}) + p(t^{n})), \tilde{e}_{u}^{n+1}\right)
\\ = &
	- \frac{1}{2}\left( \nabla (e_p^{n+1} +  e_p^n - p(t^{n+1}) + p(t^{n})), 
		{e}_{u}^{n+1} + \delta t (\nabla (e_p^{n+1} -  e_p^n) + \nabla (p(t^{n+1}) - p(t^{n})))
		\right)
\\ = &
	-\frac{\delta t}{2} (\left\| \nabla e_p^{n+1} \right\|^2 - \left\| \nabla e_p^{n} \right\|^2)
	- \delta t(\nabla (p(t^{n+1}) - p(t^{n}))), \nabla e_p^n)
	+ \frac{\delta t}{2}\left\| \nabla (p(t^{n+1}) - p(t^{n}))) \right\|^2 
\\ \leq &
	-\frac{\delta t}{2} (\left\| \nabla e_p^{n+1} \right\|^2 - \left\| \nabla e_p^{n} \right\|^2)
	+ \delta t^2\left\| \nabla e_p^{n} \right\|^2 
	+ C\left\| p\right\|^2_{W^{1,\infty}(H^1)} \delta t^2.
\end{align*}
For $J_3$, we can derive  
\begin{align*}
J_3 = &  
	\left( \nabla \cdot \sigma({\bf Q}^n, {\bf G}^{n+1}) - \nabla \cdot \sigma({\bf Q}(t^{n}), {\bf G}(t^{n+1})), \tilde{e}_{u}^{n+1}\right)
	\\ \leq &
	\left( \nabla \cdot \sigma({\bf Q}^n, e_{ G}^{n+1}) , \tilde{e}_{u}^{n+1}\right)
	- 
	\left( \sigma(e_{\bf Q}^n, {\bf G}(t^{n+1})) 
				- 2a(e_{ Q}^n:{\bf G}(t^{n+1}))e_{Q}^n, 
		\nabla \tilde{e}_{u}^{n+1}\right)
	\\ & - 2a 
	\left( 
		({\bf Q}^n:{\bf G}(t^{n+1})){\bf Q}^n - ({\bf Q}(t^{n}):{\bf G}(t^{n+1})){\bf Q}(t^{n}), 
		\nabla \tilde{e}_{u}^{n+1}\right)
\\ \leq & 
	\left( \nabla \cdot \sigma({\bf Q}^n, e_{ G}^{n+1}) , \tilde{e}_{u}^{n+1}\right)
	+ 
	C \left\| \nabla \tilde{e}_{u}^{n+1}\right\|
			\left\| G(t^{n+1}) \right\|_{L^6}
		\left\| e_{Q}^n \right\|_{L^6}
		(
		\left\| {\bf Q}^n \right\|_{L^6}
		+
		 \left\| {\bf Q}(t^n) \right\|_{L^6}
		)
	\\  &
	+ 
	C \left\| \nabla \tilde{e}_{u}^{n+1}\right\|
			\left\| G(t^{n+1}) \right\|_{L^4}
		\left\| e_{Q}^n \right\|_{L^4}
	\\ \leq &
	\underset{J_{31}}{\uwave{
	\left( \nabla \cdot \sigma({\bf Q}^n, e_{ G}^{n+1}) , \tilde{e}_{u}^{n+1}\right)
	}}
	+ \frac{\eta}{12} \left\| \nabla \tilde{e}_{u}^{n+1}\right\|^2
	+ C \left\| \nabla e_{Q}^n \right\|^2
	+ C \left\| e_{Q}^n \right\|^2.
\end{align*}
For $J_4$, we can obtain   
\begin{align*}
J_4 = &  
	\left(- \nabla {\bf Q}^n : {\bf G}^{n+1} + \nabla {\bf Q}(t^{n}) : {\bf G}(t^{n+1})
		, \tilde{e}_{u}^{n+1}\right)
	\\ = &
	-
	\left( \nabla {\bf Q}^n : e_{G}^{n+1} , \tilde{e}_{u}^{n+1}\right)
	-\left( \nabla e_{Q}^n : {\bf G}(t^{n+1}), \tilde{e}_{u}^{n+1}\right)
	\\ \leq &
	-
	\left( \nabla {\bf Q}^n : e_{G}^{n+1} , \tilde{e}_{u}^{n+1}\right)
	+\left\| \nabla e_{Q}^n \right\| \left\| {\bf G}(t^{n+1})\right\|_{L^4} 
	\left\| \tilde{e}_{u}^{n+1}\right\|_{L^4}
	\\ \leq &
	-\underset{J_{41}}{\uwave{
	\left( \nabla {\bf Q}^n : e_{G}^{n+1} , \tilde{e}_{u}^{n+1}\right)
	}}
	+ \frac{\eta}{12} \left\| \nabla \tilde{e}_{u}^{n+1}\right\|^2
	+ C \left\|\nabla e_{Q}^n \right\|^2.
\end{align*}

Finally, we also have
\begin{align*}
I_3 =& (R_Q^{n+1}, e_G^{n+1}) 
	\leq \frac{M}{10} \left\|e_G^{n+1} \right\|^2 + C \left\| R_Q^{n+1} \right\|^2, 
\\
J_5 = &  
	\left( R_u^{n+1}, \tilde{e}_{u}^{n+1}\right)
	\leq 
	\frac{\eta}{12} \left\| \nabla \tilde{e}_{u}^{n+1}\right\|^2
	+ C \left\| R_{u1}^{n+1} \right\|^2 + C \left\| R_{u2}^{n+1} \right\|^2,
\\ 
K_1 \leq & 
		|e_r^{n+1}| \left\| {\bf V}^n - {\bf V}(t^n) \right\| 
			\left\| {\bf Q} \right\|_{W^{1,\infty}(0,T, L^2)}
	\leq C |e_r^{n+1}|^2 + C \left\| {\bf V}^n - {\bf V}(t^n) \right\|^2
\\ 
	\leq &  C |e_r^{n+1}|^2 + C \left\| e_Q^{n} \right\|^2, 
\\ 
K_2 \leq & 2|e_r^{n+1}||R_r^{n+1}|
	\leq C |e_r^{n+1}|^2 + C |R_r^{n+1}|^2.
\end{align*}

\textbf{Step 4.}
Combing the above estimates, we can derive the theorem\,\ref{theorem:error_estimates}, 
which is the main result of error estimate.

In conclusion, we have
\begin{align*}
&\frac{K}{2\delta t}\left( 
	\left\| \nabla e_Q^{n+1} \right\|^2 - \left\| \nabla e_Q^{n} \right\|^2
	+ \left\| \nabla e_Q^{n+1} - \nabla e_Q^{n} \right\|^2
	\right)
	+ \frac{S_Q}{2\delta t}\left( 
	\left\| e_Q^{n+1} \right\|^2 - \left\| e_Q^{n} \right\|^2
	+ \left\| e_Q^{n+1} - e_Q^{n} \right\|^2
\right)
	+ M \left\| e_{G}^{n+1} \right\|^2 
&
\\ &
	+ \frac{1}{2\delta t}\left( 
		\left\| e_u^{n+1}\right\|^2 - \left\| e_u^{n}\right\|^2 
		+\left\| \tilde{e}_u^{n+1}-e_u^{n}\right\|^2 
	\right)
	+ \eta \left\| \nabla \tilde{e}_u^{n+1}\right\|^2
	+ \frac{1}{\delta t}\left( 
		|{e}_{r}^{n+1}|^2 - |{e}_{r}^{n}|^2 + |{e}_{r}^{n+1} - e_r^{n}|^2
	\right)
\\ \leq & 
\underset{I_{11}}{\uwave{\left( \tilde{{e}}_u^{n+1} \cdot \nabla {\bf Q}^{n} , e_G^{n+1} \right)}}
	-\underset{I_{21}}{\uwave{\left( S(\nabla \tilde{e}_u^{n+1}, {\bf Q}^{n}) , e_G^{n+1} \right)}}
		+\underset{J_{31}}{\uwave{
	\left( \nabla \cdot \sigma({\bf Q}^n, e_{ G}^{n+1}) , \tilde{e}_{u}^{n+1}\right)
	}}
		-\underset{J_{41}}{\uwave{
	\left( \nabla {\bf Q}^n : e_{G}^{n+1} , \tilde{e}_{u}^{n+1}\right)
	}}
\\ & 
	+(R_Q^{n+1}, e_G^{n+1}) + \left( R_u^{n+1}, \tilde{e}_{u}^{n+1}\right)
	+ 2e_r^{n+1} R_r^{n+1}
\\ &
	-\frac{\delta t}{2} (\left\| \nabla e_p^{n+1} \right\|^2 - \left\| \nabla e_p^{n} \right\|^2)
	+
	\frac{2M}{5} \left\| e_{G}^{n+1} \right\|^2 
	+ \frac{5\eta}{12} \left\| \nabla \tilde{e}_{u}^{n+1}\right\|^2
	+ C(1+\left\| \nabla \tilde{\bf u}^{n+1} \right\|^2)\left\| \nabla e_Q^n \right\|^2
\\ &
+ C\left( 
		\left\| e_Q^n \right\|^2
		+ \left\| {e}_{u}^{n} \right\|^2
	+ \delta t^2\left\| \nabla e_p^{n} \right\|^2 
			\right)
	 + C \left( 
		\left\| R_G^{n+1}\right\|_{H^1}^2
	 	+ C \left\| p\right\|^2_{W^{1,\infty}(H^1)} \delta t^2
	 	\right)
\\ \leq & 
	-\frac{\delta t}{2} (\left\| \nabla e_p^{n+1} \right\|^2 - \left\| \nabla e_p^{n} \right\|^2)
	+
	\frac{M}{2} \left\| e_{G}^{n+1} \right\|^2 
	+ \frac{\eta}{2} \left\| \nabla \tilde{e}_{u}^{n+1}\right\|^2
	+ C(1+\left\| \nabla \tilde{\bf u}^{n+1} \right\|^2)\left\| \nabla e_Q^n \right\|^2
\\ &
+ C\left( 
		\left\| e_Q^n \right\|^2
		+ \left\| {e}_{u}^{n} \right\|^2
		+ |e_r^{n+1}|^2
	+ \delta t^2\left\| \nabla e_p^{n} \right\|^2 
			\right)
\\ &
	 + C \left( 
	 	\left\| R_Q^{n+1}\right\|^2
		+ \left\| R_G^{n+1}\right\|_{H^1}^2
		+ \left\| R_{u1}^{n+1} \right\|^2
		+ \left\| R_{u2}^{n+1} \right\|^2
	 	+ |R_r^{n+1}|^2
	 	+ C \left\| p\right\|^2_{W^{1,\infty}(H^1)} \delta t^2
	 	\right)
&
\\ \leq &
	\frac{M}{2} \left\| e_{G}^{n+1} \right\|^2 
	+ \frac{\eta}{2} \left\| \nabla \tilde{e}_{u}^{n+1}\right\|^2
	-\frac{\delta t}{2} (\left\| \nabla e_p^{n+1} \right\|^2 - \left\| \nabla e_p^{n} \right\|^2)
&
\\ &
	+ C\left(
		(1+\left\| \nabla \tilde{\bf u}^{n+1} \right\|^2)\left\| \nabla e_Q^n \right\|^2 + \left\| e_Q^n \right\|^2
		+ \left\| {e}_{u}^{n} \right\|^2 
		+ |e_r^{n+1}|^2		
		+ \delta t^2\left\| \nabla e_p^{n} \right\|^2
	\right)
	+ C |r_{tt}|_{L^2(t^n, t^{n+1})}^2 \delta t
	+ C\delta t^2,
\end{align*}
where we have used $I_{11} = J_{41}$, $I_{21} = J_{31}$ and 
Lemma\,\ref{lemma:R_Q_G_u}.
Therefore,
\begin{align*}
& 
	\frac{K}{2}\left\| \nabla e_Q^{n+1} \right\|^2 - \frac{K}{2}\left\| \nabla e_Q^{n} \right\|^2
	+\frac{S_Q}{2}\left\|e_Q^{n+1} \right\|^2 - \frac{S_Q}{2}\left\| e_Q^{n} \right\|^2
		+\frac{1}{2}\left\| e_u^{n+1}\right\|^2 - \frac{1}{2}\left\| e_u^{n}\right\|^2 
		+|{e}_{r}^{n+1}|^2 - |{e}_{r}^{n}|^2 
&
\\ &
	+ \frac{\delta t^2}{2}(\left\| \nabla e_p^{n+1} \right\|^2 - \left\| \nabla e_p^{n} \right\|^2)
	+ \frac{\eta}{2}\delta t \left\| \nabla \tilde{e}_u^{n+1}\right\|^2
	+ \frac{M}{2}\delta t  \left\| e_{G}^{n+1} \right\|^2 
\\ \leq &
	C \delta t  \left(
		(1+\left\| \nabla \tilde{\bf u}^{n+1} \right\|^2)\left\| \nabla e_Q^n \right\|^2 + \left\| e_Q^n \right\|^2
		+ \left\| {e}_{u}^{n} \right\|^2 
		+ |e_r^{n+1}|^2		
		+ \delta t^2\left\| \nabla e_p^{n} \right\|^2
	\right)
	+ C |r_{tt}|_{L^2(t^n, t^{n+1})}^2 \delta t^2
	+ C\delta t^3.
\end{align*}
Applying the discrete Gronwall inequality\,\cite{He_Y_2007_Stability, Shen_J_1990_Long}, 
there exists a positive constant $\delta t_1$ such that
\begin{align*}
&\frac{K}{2}\left\| \nabla e_Q^{k+1} \right\|^2
			+ \frac{S_Q}{2}\left\|e_Q^{k+1} \right\|^2
			+ \frac{1}{2}\left\| e_u^{k+1}\right\|^2
			+|{e}_{r}^{k+1}|^2  
			+ \frac{\delta t^2}{2}\left\| \nabla e_p^{k+1} \right\|^2
		\\ & \qquad \qquad \qquad \qquad \qquad \qquad \qquad \qquad 
		+ \frac{\eta}{2}\delta t\sum_{n=0}^{k+1} \left\| \nabla \tilde{e}_u^{n+1}\right\|^2
	+ \frac{M}{2} \delta t \sum_{n=0}^{k+1}\left\| e_{G}^{n+1} \right\|^2 
\leq C \delta t^2.
\end{align*}
where $\delta t < \delta t_1$.

\textbf{Step 5.} Prove $\left\| {\bf Q}^{k+1} \right\|_{L^{\infty}} \leq C$.

To give the proof of $\left\| {\bf Q}^{k+1} \right\|_{L^{\infty}} \leq C$, 
we need $H^2$ boundedness of $e_Q^{k+1}$. 
With the $H^2$ regularity results for elliptic equation, 
we can obtain
\begin{align*}
\left\| e_Q^{k+1} \right\|_{H^2} 
	\leq &
	 C \left( \left\| e_Q^{k+1} \right\| + \left\| \Delta e_Q^{k+1} \right\| \right)
	 \\ \leq &
	 C \left( \left\| e_Q^{k+1} \right\| 
	 	+ \frac{1}{K}\left\| e_G^{k+1} + S_Q e_Q^{k+1} +r^{k+1} {\bf V}^k - r(t^{k+1}) {\bf V}(t^k) - R_G^{k+1} \right\| \right)
	 \\ \leq &
	 C \left( \left\| e_Q^{k+1} \right\| 
	 	+ \left\| e_G^{k+1} \right\| 
	 	 + \left\| r^{k+1} {\bf V}^k - r(t^{k+1}) {\bf V}(t^k)\right\| 
	 	 + \left\| R_G^{k+1} \right\| \right)
	 \\ \leq & C \delta t^{\frac{1}{2}},
\end{align*}
Then, we can derive
\begin{align*}
\left\| {\bf Q}^{k+1} \right\|_{L^{\infty}} 
	\leq \left\| e_{Q}^{k+1} \right\|_{L^{\infty}} 
			+ \left\| {\bf Q}(t^{k+1}) \right\|_{L^{\infty}}
	\leq  \left\| e_{Q}^{k+1} \right\|_{H^2} 
			+ \left\| {\bf Q}(t^{k+1}) \right\|_{L^{\infty}}
	\leq C_1 \delta t^{\frac{1}{2}} + \left\| {\bf Q}(t^{k+1}) \right\|_{L^{\infty}}
	\leq \mathcal{C}.
\end{align*}
where $\delta t \leq \delta t_2 = 1/C_1^2$.
Thus, we obtain the conclusion for $\delta t < \delta t_0 = min\{ \delta t_1, \delta t_2 \}$.
\end{proof}

\begin{theorem}\label{theorem:error_estimates}
Assume the assumption \eqref{assumptions_for_exact_solution} holds,
we have
\begin{align}
&\frac{K}{2}\left\| \nabla e_Q^{k+1} \right\|^2
			+ \frac{S_Q}{2}\left\|e_Q^{k+1} \right\|^2
			+ \frac{1}{2}\left\| e_u^{k+1}\right\|^2
			+|{e}_{r}^{k+1}|^2  
			+ \frac{\delta t^2}{2}\left\| \nabla e_p^{k+1} \right\|^2
\notag \\ & \qquad \qquad \qquad 
		+ \frac{\eta}{2}\delta t\sum_{n=0}^{k+1} \left\| \nabla \tilde{e}_u^{n+1}\right\|^2
	+ \frac{M}{2} \delta t \sum_{n=0}^{k+1}\left\| e_{G}^{n+1} \right\|^2 
\leq C \delta t^2, 
\qquad 
\forall \ 0 \leq k \leq N-1.
\label{eq:error_estimate}
\end{align}
\end{theorem}
\begin{proof}
By the proof of Lemma \ref{lemma:infty}, 
	Eq.\,\eqref{eq:error_estimate} holds for $\delta t < \delta t_0$. 
If $\delta t \geq \delta t_0$, 
	we have
	\begin{align*}
	&\frac{K}{2}\left\| \nabla e_Q^{k+1} \right\|^2
			+ \frac{S_Q}{2}\left\|e_Q^{k+1} \right\|^2
			+ \frac{1}{2}\left\| e_u^{k+1}\right\|^2
			+|{e}_{r}^{k+1}|^2  
			+ \frac{\delta t^2}{2}\left\| \nabla e_p^{k+1} \right\|^2
\notag \\ & \qquad \qquad \qquad \qquad \qquad
		+ \frac{\eta}{2}\delta t\sum_{n=0}^{k+1} \left\| \nabla \tilde{e}_u^{n+1}\right\|^2
	+ \frac{M}{2} \delta t \sum_{n=0}^{k+1}\left\| e_{G}^{n+1} \right\|^2 
\leq C_2 \leq C_2 (\delta t_0)^{-2} \delta t^2, 
	\end{align*}
where $k = 0,1,2,...,N-1$.
Finally, we can obtain Eq.\,\eqref{eq:error_estimate} holds for any $\delta t$. 
\end{proof}
\textbf{Remark.} 
Note that our proof is stated for the boundary condition \eqref{boundary:dirichlet},
but for the periodic boundary we can get the same conclusion
due to the following estimation 
\begin{align*}
\left\| \tilde{e}_u^{n+1} \right\|^2 
	\leq 
	\left\| e_u^{n+1} \right\|^2 
	+ \delta t^2 \left( \left\| \nabla e_p^{n+1} \right\|^2 + \left\| \nabla e_p^{n} \right\|^2 \right)
	+ C \delta t^4.
\end{align*}

\section{Numerical examples}
\label{Sec:Numerical_examples}
In this section, 
	several numerical examples are given to verify the effectiveness of our scheme \eqref{eq:first_order_scheme_Q}-\eqref{eq:first_order_scheme_r}.
Firstly, 
mesh refinement tests are designed to validate the convergence rate of 
	$L^2$-norm error of our proposed method in time direction. 
Secondly, the dissipation property of the modified energy is validated numerically.
In numerical simulations below, 
we assume that the directors of the LCPs are imposed in the $x-y$ plane, 
so that the tensor ${\bf Q}$ is a $2\times 2$ matrix. 
The unknown functions depend on
variables $x$ and $y$ only and the model parameters are 
\begin{align*}
a = 1, \quad \eta = 1, \quad \alpha = -0.2, \quad \gamma = 1,
\quad K = 0.001, \quad M = 1, \quad C_0 = 10, \quad S_Q = 30.  
\end{align*}
To ignore the spatial discrete error, 
	we employ $512 \times 512$ mesh grids on domain $\Omega = [0, 1]^2$.
We present the $L^2$ errors of the Cauchy sequence at $t = 0.1$ with different time steps $\delta t_k = 8\times {10}^{-5}/2^k$\,($k = 0,1,\cdots,5$). 
The Cauchy error between two different time step sizes $\delta t$ and $\frac{\delta t}{2}$ 
is measured by $\left\| e_f \right\| = \left\| f_{\delta t} - f_{\delta t/2}\right\|$.
\subsection{Time step refinements}
In the first examples, 
	we choose the initial and boundary conditions
	\begin{equation}\label{eq:accuracy_test_initial_1}
	\begin{aligned}
	&
	{\bf u}_0 = 0, \qquad 
		{\bf Q}_0 = {\bf n}_0 {\bf n}_0^T - \frac{\left\| {\bf n}_0 \right\|^2}{2}, \\
	&
	{\bf u}|_{\partial \Omega} = 0, \qquad	
		{\bf Q}|_{\partial \Omega} = {\bf Q}_0|_{\partial \Omega},
	\end{aligned}
	\end{equation}
where ${\bf n}_0 = (\sin(2\pi x) \sin(2\pi y), 0)^T$. 
As can be seen in Table \ref{tb:initial_1}, 
the scheme \eqref{eq:first_order_scheme_Q}-\eqref{eq:first_order_scheme_r} presents first order convergence rates for all variables.

\begin{table}[H]
\caption{Temporal $L^2$-errors and convergence rates of ${\bf Q}$, ${\bf u}$ and $r$ at $t = 0.1$ 
			with initial condition\,\eqref{eq:accuracy_test_initial_1} and $\delta t_k = 8\times {10}^{-5}/2^k$, $k = 0,1,\cdots,4$.}
\label{tb:initial_1}
\begin{center}
\setlength{\tabcolsep}{1.0mm}{
\begin{tabular}{|c|c|c|c|c|c|c|c|c|c|c|}
\hline
$\delta t$
&$Q_{11}$ & Order 
&$Q_{12}$ & Order
& $u$ & Order 
& $v$ & Order 
& $r$ & Order 
\\
\hline
$\delta t_0$
&3.03E-03   & /
&3.42E-04	& /
&1.03E-04	& /
&1.05E-04	& /
&6.82E-06  & /
\\
\hline
$\delta t_1$
&1.52E-03	& 1.00 
&1.71E-04	& 0.99
&5.16E-05	& 1.00 
&5.27E-05	& 1.00 
&3.42E-06  &  1.00 
\\
\hline
$\delta t_2$
&7.60E-04	& 1.00 
&8.59E-05	& 1.00 
&2.58E-05	& 1.00 
&2.64E-05	& 1.00 
&1.71E-06  &  1.00 
\\
\hline
$\delta t_3$
&3.80E-04	& 1.00 
&4.30E-05	& 1.00 
&1.29E-05	& 1.00 
&1.32E-05	& 1.00 
&8.56E-07  &  1.00 
\\
\hline
$\delta t_4$
&1.90E-04	& 1.00 
&2.15E-05	& 1.00 
&6.46E-06	& 1.00 
&6.60E-06	& 1.00 
&4.28E-07  &  1.00 
\\
\hline
\end{tabular}}
\end{center}
\end{table}

In the second example, 
we choose ${\bf n}_0 = (\sin(2\pi x) \sin(2\pi y), \cos(2\pi x) \cos(2\pi y))^T$
	and consider the same computational domain and mesh size. 
As can be seen in Table \ref{tb:initial_2}, 
the convergence order of scheme \eqref{eq:first_order_scheme_Q}-\eqref{eq:first_order_scheme_r} 
	is also consistent with theoretical analysis.
\begin{table}[H]
\caption{Temporal $L^2$-errors and convergence rates of ${\bf Q}$, ${\bf u}$ and $r$ at $t = 0.1$ 
			with ${\bf n}_0 = (\sin(2\pi x) \sin(2\pi y), \cos(2\pi x) \cos(2\pi y))^T$ and $\delta t_k = 8\times {10}^{-5}/2^k$, $k = 0,1,\cdots,4$.}
\label{tb:initial_2}
\begin{center}
\setlength{\tabcolsep}{1.0mm}{
\begin{tabular}{|c|c|c|c|c|c|c|c|c|c|c|}
\hline
$\delta t$
&$Q_{11}$ & Order 
&$Q_{12}$ & Order
& $u$ & Order 
& $v$ & Order 
& $r$ & Order 
\\
\hline
$\delta t_0$
&4.89E-03   & /
&2.28E-03	& /
&6.49E-05	& /
&6.49E-05	& /
&2.93E-05   & /
\\
\hline
$\delta t_1$
&2.45E-03   & 1.00
&1.14E-03	& 1.00
&3.25E-05	& 1.00
&3.25E-05	& 1.00
&1.47E-05   & 1.00
\\
\hline
$\delta t_2$
&1.23E-03   & 1.00
&5.72E-04	& 1.00
&1.62E-05	& 1.00
&1.62E-05	& 1.00
&7.36E-06   & 1.00
\\
\hline
$\delta t_3$
&6.14E-04   & 1.00
&2.86E-04	& 1.00
&8.12E-06	& 1.00
&8.12E-06	& 1.00
&3.68E-06   & 1.00
\\
\hline
$\delta t_4$
&3.07E-04   & 1.00
&1.43E-04	& 1.00
&4.06E-06	& 1.00
&4.06E-06	& 1.00
&1.84E-06   & 1.00
\\
\hline
\end{tabular}}
\end{center}
\end{table}

\subsection{Energy dissipation}
In this experiment, 
the development of defects are simulated on a period domain.
For the domain $[0, L_x]\times [0, L_y]$, 
	$L_x = L_y = 1$, 
	we consider the initial condition
	\begin{align*}
	{\bf Q}_0 = \frac{{\bf n}_0 {\bf n}_0^T}{\left\| {\bf n}_0 \right\|^2} - \frac{\bf I}{2}, 
	\quad
	{\bf n}_0 = ( x-0.25L_x, y-0.25L_y )^T.
	\end{align*}
\begin{figure}[!htbp]                                                           
    \centering      
   \begin{minipage}{0.3\linewidth}                                                        
   \centerline{\includegraphics[scale=0.26]{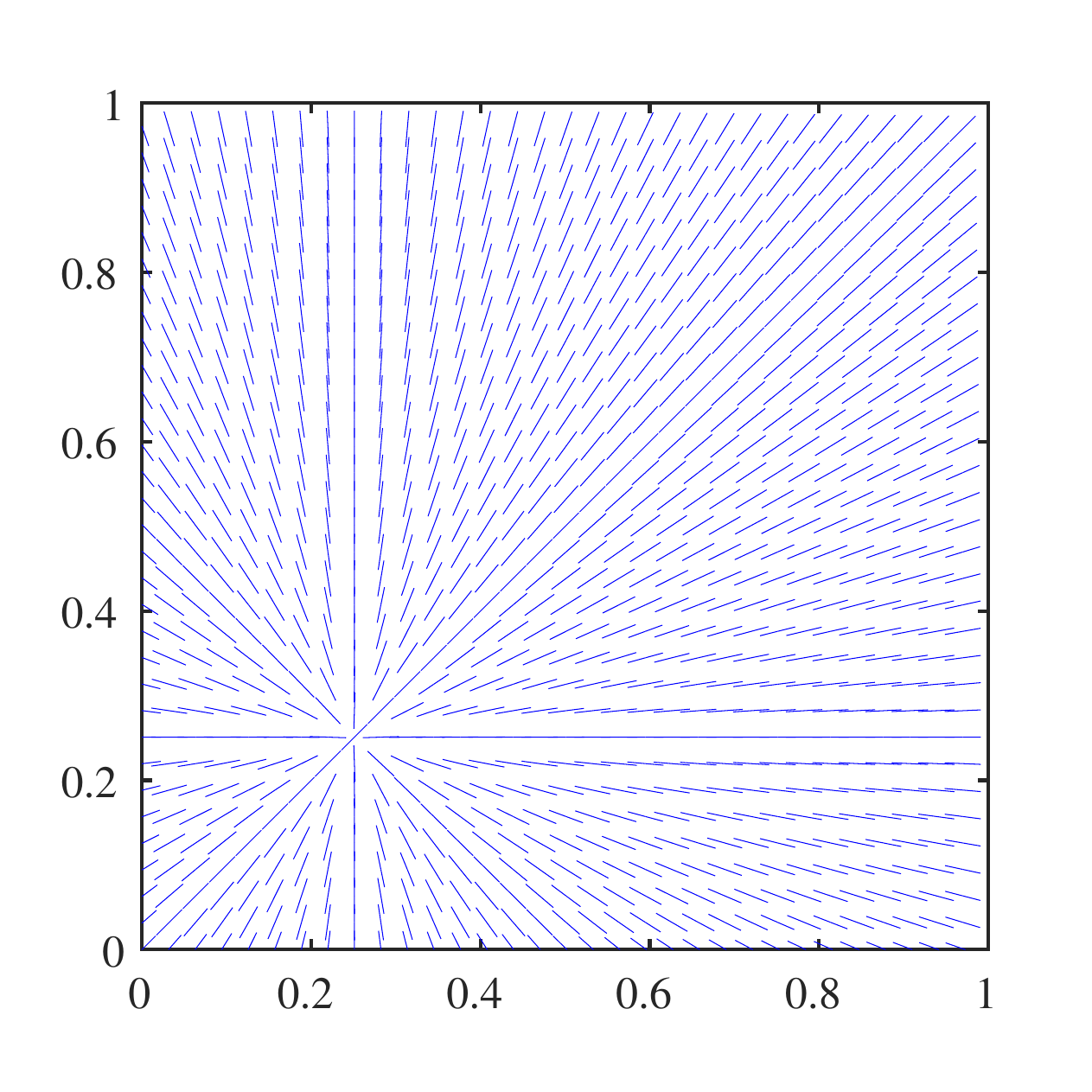}}
   \footnotesize{\centerline{(a) t = 0}} 
   \end{minipage}
   \begin{minipage}{0.3\linewidth}                                                        
   \centerline{\includegraphics[scale=0.26]{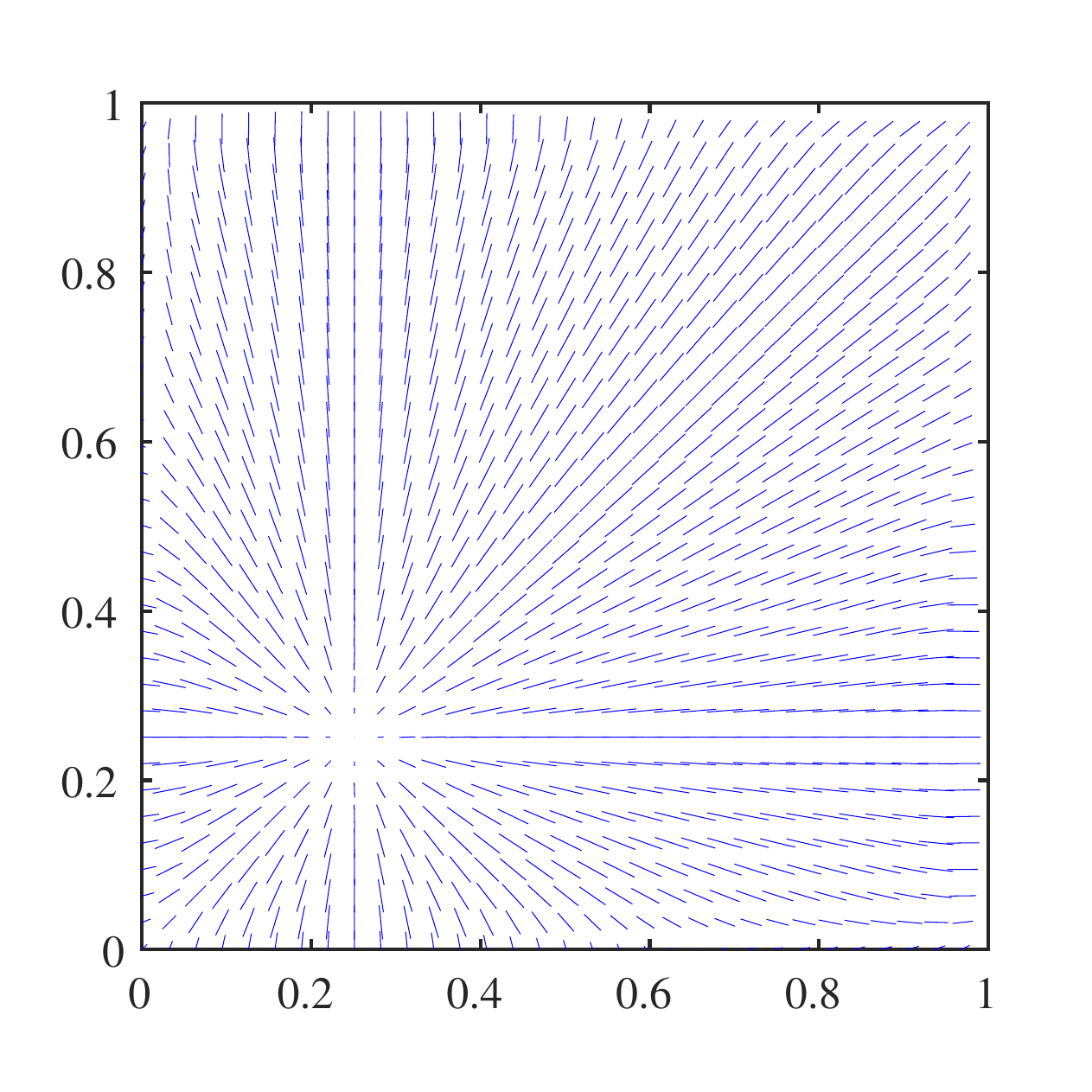}}
   \footnotesize{\centerline{(b) t = 1}}  
   \end{minipage} 
   \begin{minipage}{0.3\linewidth}                                                        
   \centerline{\includegraphics[scale=0.26]{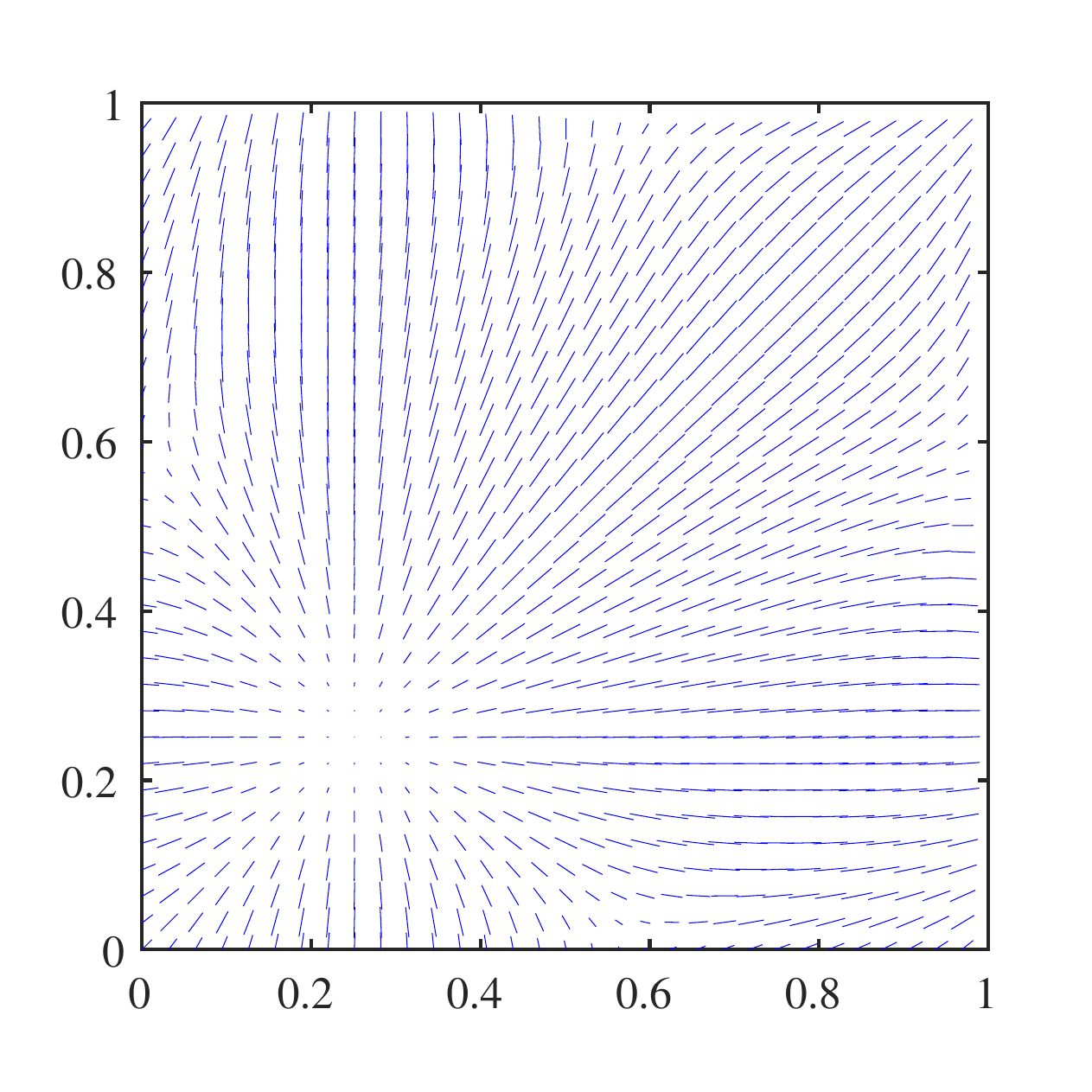}}  
   \footnotesize{\centerline{(c) t = 10}}  
   \end{minipage} 
   \begin{minipage}{0.3\linewidth}                                                        
   \centerline{\includegraphics[scale=0.26]{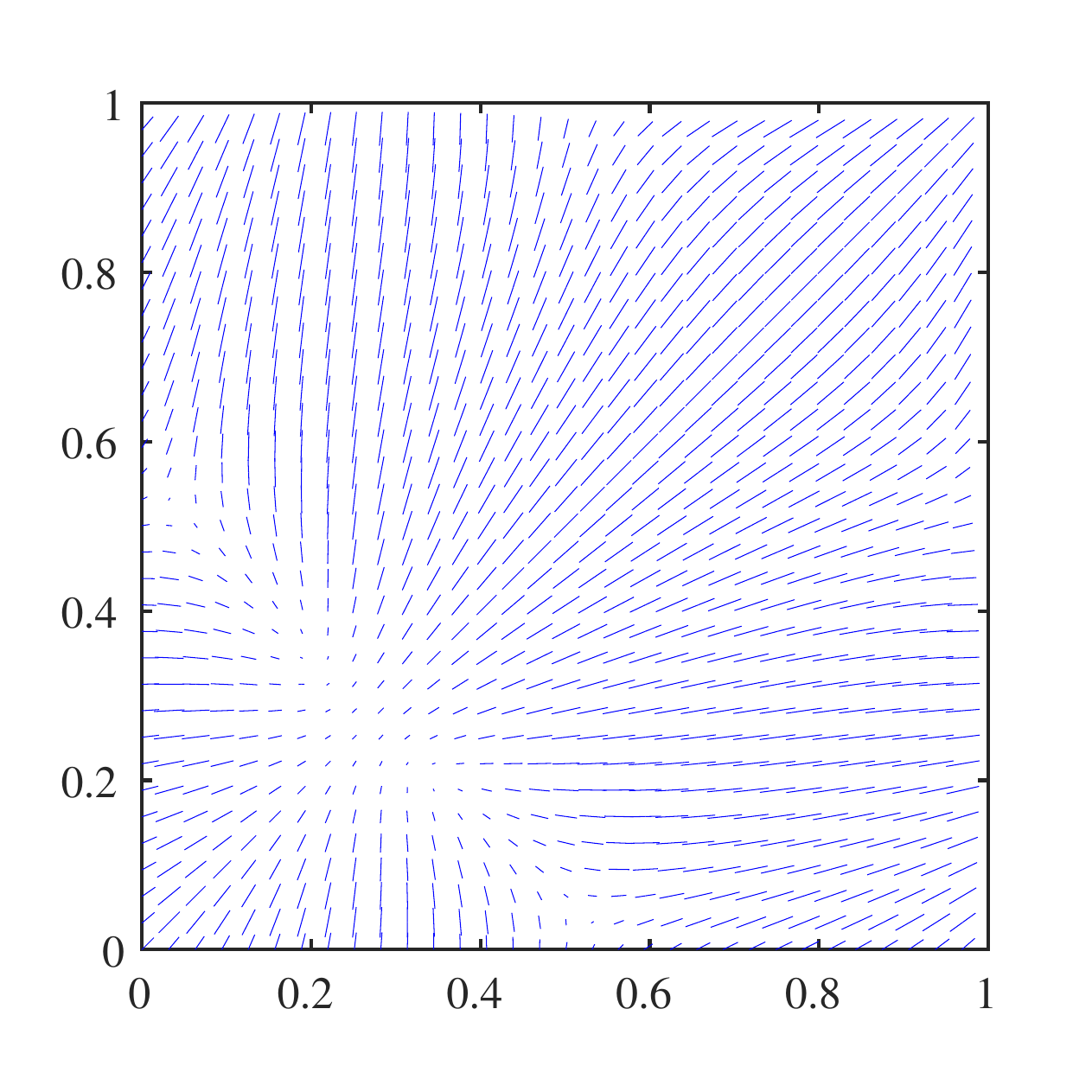}}  
   \footnotesize{\centerline{(d) t = 20}}  
   \end{minipage}  
   \begin{minipage}{0.3\linewidth}                                                        
   \centerline{\includegraphics[scale=0.26]{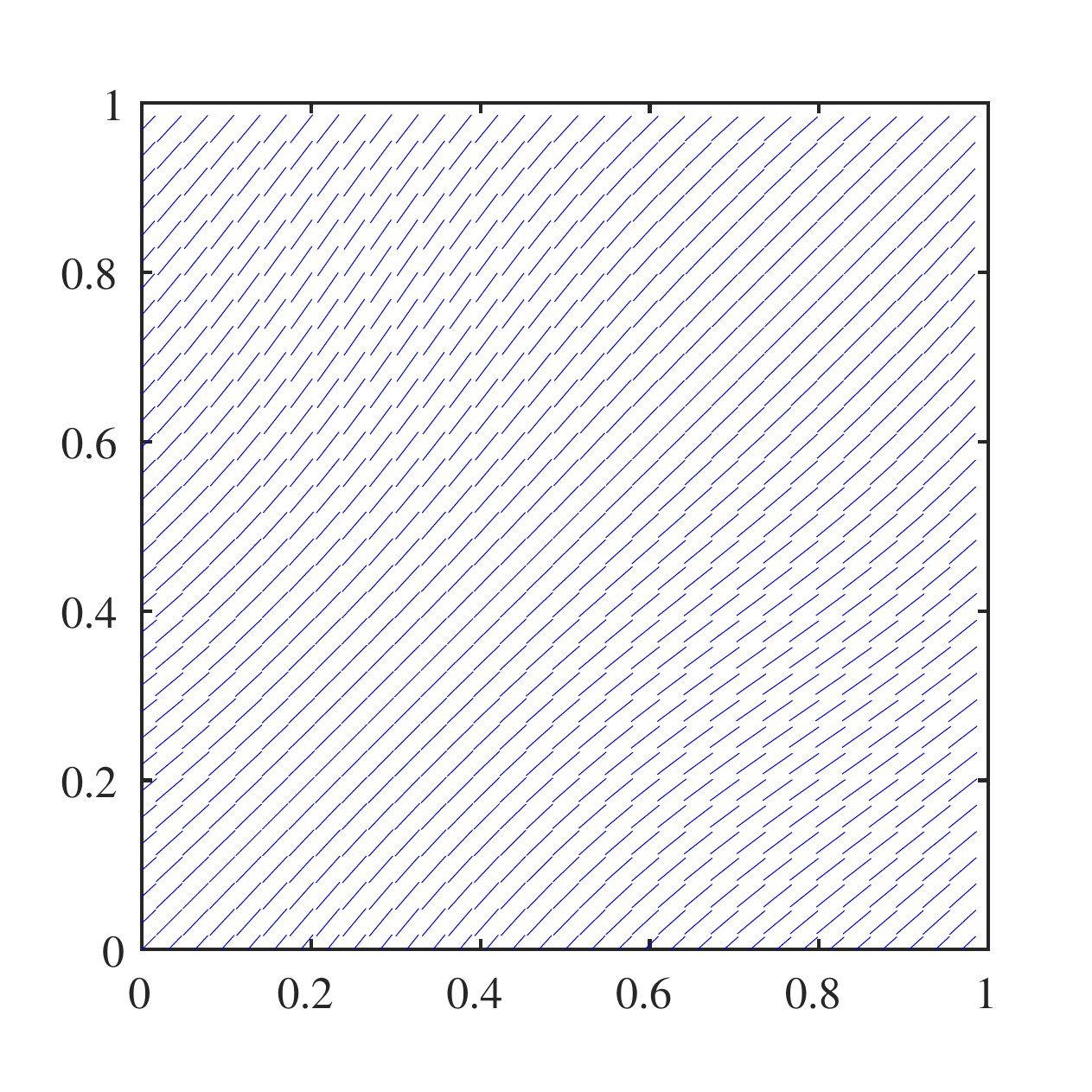}} 
   \footnotesize{\centerline{(e) t = 60}}   
   \end{minipage}
   \begin{minipage}{0.3\linewidth}                                                        
   \centerline{\includegraphics[scale=0.26]{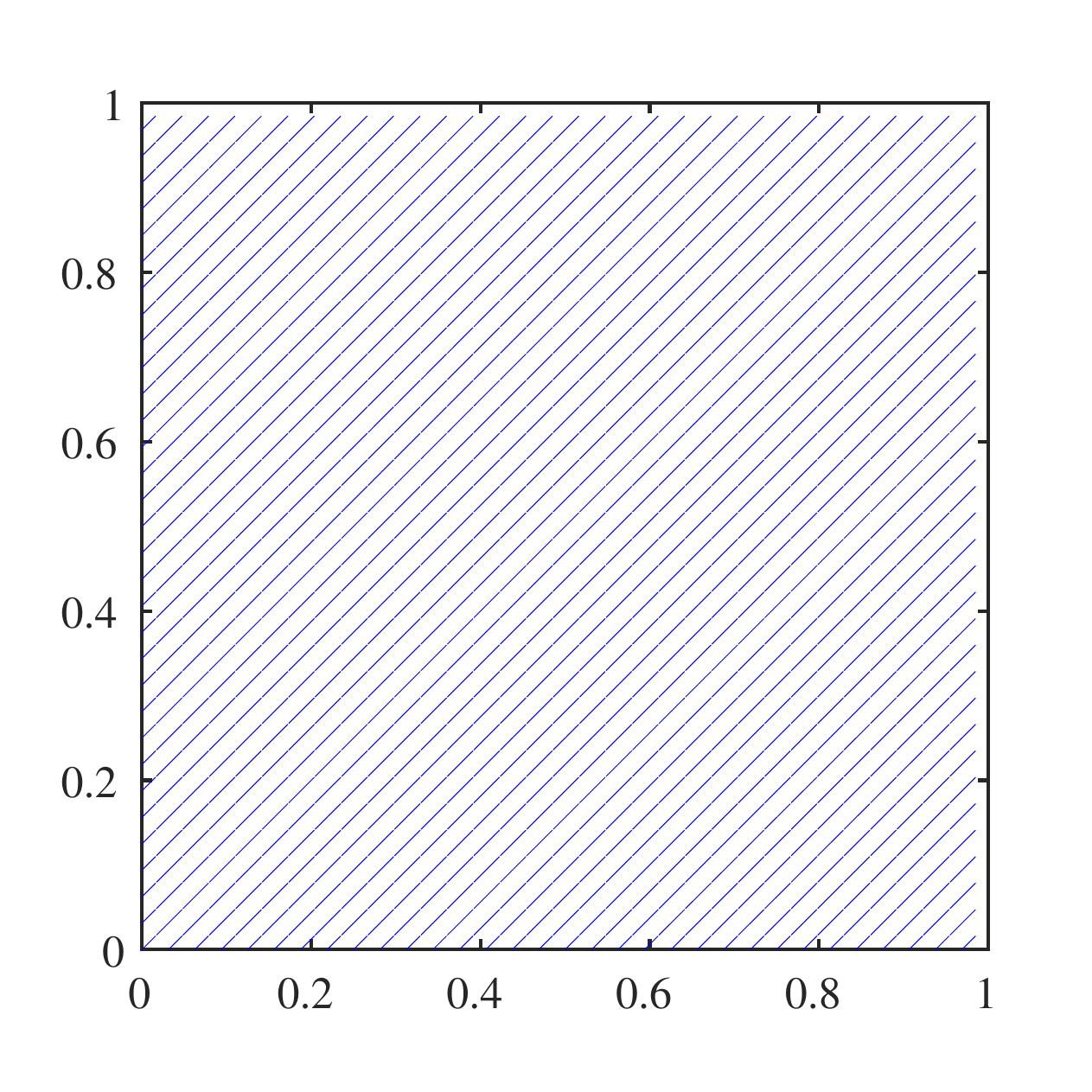}} 
   \footnotesize{\centerline{(f) t = 200}}   
   \end{minipage}   
    \caption{
    Dynamical evolution of director.
    }   
    \label{png:time_evolution}                             
\end{figure}

\begin{figure}[!htbp]                                                           
    \centering      
   \begin{minipage}[t]{1\linewidth}                                                        
   \centerline{\includegraphics[scale=0.5]{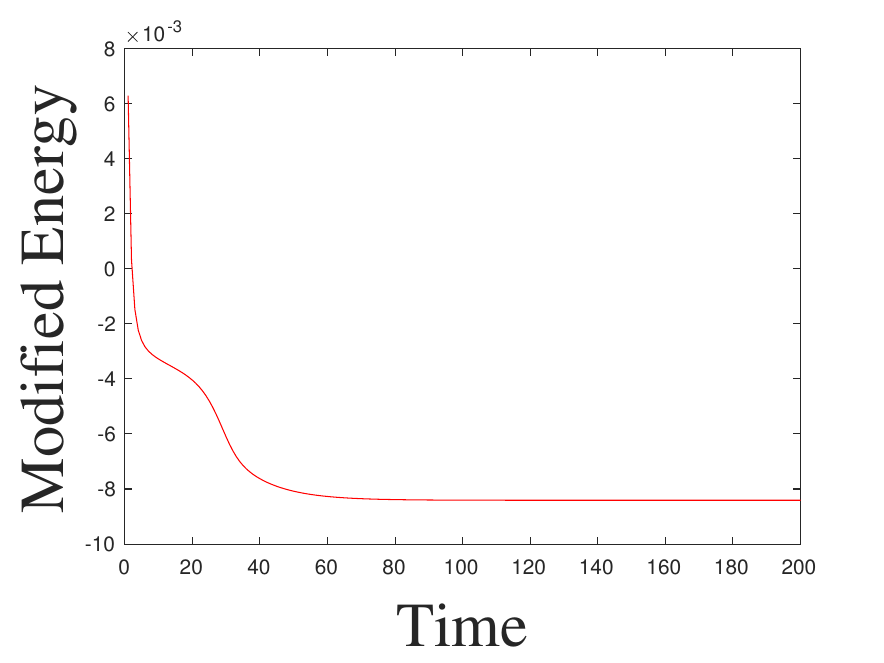}}  
   \end{minipage}   
    \caption{
    Time evolution of modified energy.
     }   
    \label{png:energy_evolution}                             
\end{figure}
Fig.\,\ref{png:time_evolution}(a) shows an initial $+1$ type defect at the point $(0.25,0.25)$. 
In Fig.\,\ref{png:time_evolution}(b-f),
the initial $+1$ type defect gradually escapes. 
Eventually, the system reaches a steady state.
Fig.\,\ref{png:energy_evolution} shows that the modified energy decreases with
time, consistent with our analysis.

\section{Conclusion}
\label{sec:conclusion}
We carry out a rigorous error analysis for a first-order SAV scheme 
	of 3D hydrodynamic ${\bf Q}$-tensor model of nematic liquid crystals.
The scheme maintains the properties of energy stability and uniqueness of solvability, 
which are proved in detail.
Thus, we derive the rigorous error estimate of order $O(\delta t)$ in the sense of $L^2$ norm.
Finally, some numerical simulations are given to demonstrate the theoretical analysis.

\section*{Acknowledgments}
G. Ji is partially supported by the National Natural Science Foundation of China (Grant
No. 12471363), X. Li is partially supported by Beijing Normal University grant 312200502508.

\nocite{*}
\bibliographystyle{unsrt}
\normalem
\bibliography{AC_NS}

\end{document}